\documentclass{article}

\usepackage[accepted]{icml2024}

\usepackage{ifthen}
\newboolean{cameraready}
\setboolean{cameraready}{true}

\usepackage[setbb]{kmath}
\usepackage{mathrsfs}
\usepackage{enumitem}
\usepackage{glossaries}
\usepackage{graphicx}
\usepackage{booktabs}
\usepackage[hidelinks]{hyperref}
\usepackage{amsmath}
\usepackage{amssymb}
\usepackage{stmaryrd}
\usepackage{mathtools}
\usepackage{amsthm}
\usepackage{lipsum}
\usepackage{soul}
\usepackage{pgfplots}
\usepackage[algo2e]{algorithm2e} 
\usepackage[capitalize,noabbrev]{cleveref}
\usepackage{bold-extra}
\usepackage{tabularx}

\theoremstyle{plain}
\newtheorem{proposition}{Proposition}

\newtheorem{definition}{Definition}

\newtheorem{lemma}{Lemma}
\pgfplotsset{compat=newest}
\usepgfplotslibrary{groupplots}
\usetikzlibrary{matrix}
\usetikzlibrary{arrows.meta}
\usepgfplotslibrary{fillbetween}
\usetikzlibrary{fit}

\pgfplotsset{
    ylabel right/.style={
        after end axis/.append code={
            \node [rotate=270, anchor=north] at (rel axis cs:1.35,0.5) {#1};
        }   
    }
}

\definecolor{mosekcolor}{RGB}{41,128,185}
\definecolor{cplexcolor}{RGB}{46,204,113}
\definecolor{gurobicolor}{RGB}{241,196,15}
\definecolor{l0bnbcolor}{RGB}{230,126,34}
\definecolor{bnbcolor}{RGB}{192,57,43}

\newcommand{\ppb}{\mathscr{P}}
\newcommand{\rpb}{\mathscr{R}}

\newcommand{\pobj}{p}
\newcommand{\robj}{r}
\newcommand{\dobj}{d}

\newcommand{\dfunc}{D}

\newcommand{\pvletter}{x}
\newcommand{\bvletter}{z}
\newcommand{\dvletter}{u}
\newcommand{\vvletter}{v}
\newcommand{\pv}{\mathbf{\pvletter}}
\newcommand{\bv}{\mathbf{\bvletter}}
\newcommand{\dv}{\mathbf{\dvletter}}
\newcommand{\vv}{\mathbf{\vvletter}}
\newcommand{\pvi}[1]{\pvletter_{#1}}

\newcommand{\vvi}[1]{\vvletter_{#1}}

\newcommand{\pdim}{n}
\newcommand{\ddim}{m}
\newcommand{\dic}{\mathbf{A}}
\newcommand{\atom}[1]{\mathbf{a}_{#1}}
\newcommand{\reg}{\lambda}
\newcommand{\lossfunc}{f}
\newcommand{\pertfunc}{h}
\newcommand{\regfunc}{g}
\newcommand{\pset}{\mathcal{X}}
\newcommand{\genericfunc}{\omega}

\newcommand{\setidx}{\mathcal{S}}
\newcommand{\idxentry}{i}

\newcommand{\oneSymb}{1}
\newcommand{\zeroSymb}{0}
\newcommand{\noneSymb}{\bullet}
\newcommand{\setzero}{\setidx_{\zeroSymb}}
\newcommand{\setone}{\setidx_{\oneSymb}}
\newcommand{\setnone}{\setidx_{\noneSymb}}
\newcommand{\setidxzero}{\mathcal{I}_0}
\newcommand{\setidxone}{\mathcal{I}_1}

\newcommand{\nodeSymb}{\nu}
\newcommand{\node}[1]{#1^\nodeSymb}

\newcommand{\succnode}[2]{\nodeSymb_{#1,#2}}

\newcommand{\pivot}[1]{\Delta_{#1}}

\newcommand{\sparsitylevel}{k}
\newcommand{\groundtruth}{\pv^{\dagger}}

\newcommand{\obs}{\mathbf{y}}
\newcommand{\bigM}{M}
\newcommand{\noise}{\mathbf{e}}
\newcommand{\corrmat}{\mathbf{\Sigma}}
\newcommand{\corrmatel}[1]{\Sigma_{#1}}
\newcommand{\corrparam}{\rho}
\newcommand{\snr}{\tau}

\DeclareMathOperator{\argmax}{argmax}

\newcommand{\1}{\mathbf{1}}
\newcommand{\0}{\mathbf{0}}
\newcommand{\bigO}{\mathcal{O}}
\newcommand{\intervint}[2]{\llbracket#1,#2\rrbracket}
\newcommand{\LB}[1]{\tilde{#1}}
\newcommand{\UB}[1]{\bar{#1}}
\newcommand{\card}[1]{|#1|}
\newcommand{\transp}[1]{#1^{\mathrm{T}}}
\newcommand{\opt}[1]{#1^{\star}}
\newcommand{\conj}[1]{#1^{\star}}
\newcommand{\biconj}[1]{#1^{\star\star}}
\newcommand{\abs}[1]{|#1|}
\newcommand{\icvx}{\eta}
\newcommand{\norm}[2]{\|#1\|_#2}
\newcommand{\pospart}[1]{[#1]_+}

\newcommand{\subdiff}{\partial}
\newcommand{\separable}[2]{#1_{#2}}
\newcommand{\dom}{\mathrm{dom}\,}

\newcommand{\proofset}{\mathcal{C}}
\newcommand{\prooffun}{\omega}

\newcommand{\eg}{\textit{e.g.}}

\ifthenelse{\boolean{cameraready}}{
  \newcommand{\AddNote}[1]{}
  \newcommand{\AddTodo}[1]{}

  \newcommand{\ProvideEditionMacros}[2]{%
    \expandafter\NewDocumentCommand\csname Add#1\endcsname{gg}{##1}%
    \expandafter\NewDocumentCommand\csname Rem#1\endcsname{gg}{}%
    \expandafter\NewDocumentCommand\csname Sup#1\endcsname{gg}{}%
  }
}{
  \newcommand{\AddNote}[1]{\textcolor{red}{[Note: #1]}}
  \newcommand{\AddTodo}[1]{\textcolor{red}{[Todo: #1]}}

  \newcommand{\ProvideEditionMacros}[2]{%
    \expandafter\NewDocumentCommand\csname Add#1\endcsname{gg}{\textcolor{#2}{##1}}%
    \expandafter\NewDocumentCommand\csname Rem#1\endcsname{gg}{\textcolor{#2}{[#1: \small ##1]}}%
    \expandafter\NewDocumentCommand\csname Sup#1\endcsname{gg}{\textcolor{#2}{\st{##1}}}%
  }
}
\newacronym{bnb}{BnB}{Branch-and-Bound}
\newacronym{mip}{MIP}{Mixed Integer Program}

\ProvideEditionMacros{TG}{blue}
\ProvideEditionMacros{CH}{orange}
\ProvideEditionMacros{CE}{red}
\ProvideEditionMacros{AA}{green}

\newcommand{\thepapertitle}{A New Branch-and-Bound Pruning Framework for $\ell_0$-Regularized Problems}

\icmltitlerunning{\thepapertitle}

\begin{document}

\twocolumn[%
\icmltitle{\thepapertitle}

\begin{icmlauthorlist}
    \icmlauthor{Théo Guyard}{tg}
    \icmlauthor{Cédric Herzet}{ch}
    \icmlauthor{Clément Elvira}{ce}
    \icmlauthor{Ay\c{s}e-Nur Arslan}{aa}
\end{icmlauthorlist}

\icmlaffiliation{tg}{Inria and Insa, Univ Rennes, CNRS, IRMAR - UMR 6625, Rennes, France}
\icmlaffiliation{ch}{Ensai, Univ Rennes, CNRS, CREST - UMR 9194, Rennes, France}
\icmlaffiliation{ce}{CentraleSupélec, Univ Rennes, CNRS, IETR - UMR 6164, Rennes, France}
\icmlaffiliation{aa}{Inria, Univ Bordeaux, CNRS, IMB - UMR 5251, Bordeaux, France}

\icmlcorrespondingauthor{Théo Guyard}{theo.guyard@insa-rennes.fr}

\icmlkeywords{L0-regularization, Branch-and-Bound, Pruning test}

\vskip 0.3in
]

\printAffiliationsAndNotice{}

\begin{abstract}
    We consider the resolution of learning problems involving $\ell_0$-regularization via \gls{bnb} algorithms.~These methods explore regions of the feasible space of the problem and check whether they do not contain solutions through ``pruning tests''.
    In standard implementations, evaluating a pruning test requires to solve a convex optimization problem, which may result in computational bottlenecks.
    In this paper, we present an alternative to implement pruning tests for some generic family of $\ell_0$-regularized problems.~Our proposed procedure allows the simultaneous assessment of several regions and can be embedded in standard \gls{bnb} implementations with a negligible computational overhead. We show through numerical simulations that our pruning strategy can improve the solving time of \gls{bnb} procedures by several orders of magnitude for typical problems encountered in machine-learning applications. 
\end{abstract}


\section{Introduction}
\label{sec:intro}

This paper focuses on optimization problems of the form: 
\begin{equation}
    \stepcounter{equation}
    \tag{$\mathscr{P}$}
    \label{prob:prob}
    \opt{\pobj} = \textstyle\inf_{\pv \in \kR^{\pdim}} \lossfunc(\dic\pv) 
    + \regfunc(\pv)
\end{equation}
where $\lossfunc: \kR^{\ddim} \mapsto \kR\cup\{+\infty\}$ is a loss function, $\dic\in \kR^{\ddim\times\pdim}$ is a given matrix and $\regfunc: \kR^{\pdim} \mapsto \kR\cup\{+\infty\}$ is a regularization function expressed as
\begin{equation}
    \label{eq:regfunc}
    \textstyle
    \regfunc(\pv) = \reg\norm{\pv}{0} + \sum_{\idxentry=1}^\pdim\pertfunc(\pvi{\idxentry})
\end{equation}
for some $\reg > 0$ and $\pertfunc: \kR \mapsto \kR \cup \{+\infty\}$.
On the one hand, the so-called ``$\ell_0$-norm'' is defined as
\begin{equation*}
    \norm{\pv}{0} = \mathrm{card}(\kset{\idxentry \in \intervint{1}{\pdim}}{\pvi{\idxentry}\neq 0})
\end{equation*}
and promotes sparsity in the optimizers of \eqref{prob:prob} by counting the number of non-zero elements in its argument.
On the other hand, the term $\pertfunc(\cdot)$ allows to enforce additional application-specific properties, see \eg, \cite{bruer2015designing,bertsimas2021unified}.
Solving problem~\eqref{prob:prob} is of interest in many fields including machine learning, high-dimensional statistics or signal processing.
This problem is for instance linked to feature selection \cite{bertsimas2016best}, compressive sensing \cite{candes2007sparse}, principal component analysis \cite{bertsimas2022solving}, sparse SVM \cite{tan2010learning} or neural network pruning \cite{carreira2018learning}, among others.
The reader can refer to \cite{tillmann2021cardinality,bertsimas2021unified} for an extensive review of related applications.

Since~\eqref{prob:prob} is NP-hard~\cite{nguyen2019np}, the main trends of work in the last decades have focused on addressing relaxed instances of this problem or inferring its solutions through heuristic procedures~\cite{tropp2010computational}.
However, it has recently been emphasized that the solutions of the original problem \eqref{prob:prob} may enjoy much better statistical properties than those obtained by these sub-optimal strategies~\cite{bertsimas2020sparse,zhong2022L}. 
Consequently, there has recently been a revived interest in solving \eqref{prob:prob} exactly and several studies have emphasized that discrete-optimization tools can sometimes provide tractable solutions, see \eg~\cite{bertsimas2016best}.
In this vein, state-of-the-art procedures are mostly based on \gls{bnb} algorithms whose process can be specialized to exploit the structure of \eqref{prob:prob} and achieve competitive running times \cite{ben2022global,hazimeh2022sparse}.

In a nutshell, \gls{bnb} algorithms solve an optimization problem by successively: 
\textit{i)} dividing the feasible space into regions and 
\textit{ii)} trying to detect regions that cannot contain a minimizer. 
This second step is commonly referred to as ``pruning test'' and is based on the construction of some lower bounds on the value that the objective function can take.
A standard approach to constructing these lower bounds is based on the minimization of a convex lower-approximation of the objective function of~\eqref{prob:prob}, called ``relaxation''.
This operation usually dominates the complexity of \gls{bnb} algorithms and can lead to tractability issues for some problem instances. 
In this work, we propose a strategy to alleviate this computational bottleneck.

\subsection{Contributions}
\label{sec:intro:contributions}

We present a novel methodology to implement pruning tests in \gls{bnb} algorithms that does not leverage the solution of a convex optimization problem and show that it can be embedded at virtually no cost in standard \gls{bnb} implementations. 
Numerical simulations reveal that our strategy can improve the solving time by several orders of magnitude.
Our contribution follows some recent lines of work \cite{atamturk2020safe,guyard2022node} and exploits Fenchel-Rockafellar duality~\cite{rockafellar1967duality}.
In contrast to these prior contributions which focused on specific\footnote{
    Prior contributions focused on instances defined by specific choices of the functions $\lossfunc(\cdot)$ and $\pertfunc(\cdot)$.
    These works referred to their methodology as ``screening''.
    In this paper, we rather use the terminology ``pruning'' which is standard in the \gls{bnb} literature.
} instances of problem \eqref{prob:prob}, we introduce a general framework encompassing a large family of problems typically encountered in machine learning. 
Specifically, our framework applies under the following set of hypotheses:
\begin{enumerate}[topsep=2pt,itemsep=3pt,parsep=0pt, label=({$\mathrm{H}_\arabic*$}),leftmargin=*]
    \item The function $\lossfunc(\cdot)$ is proper, closed and convex.\label{hyp:lossfunc}
    \item The function $\pertfunc(\cdot)$ is proper, closed and convex.\label{hyp:pertfunc}
    \item $\pvi{}=0$ is an accumulation point\footnote{
        \(\pvi{}\) is said to be an accumulation point of a set \(\proofset\subseteq\kR\) if for all neighborhoods \(\mathcal{N}\) of \(\pvi{}\), the set \(\mathcal{N}\cap\proofset\setminus\{\pvi{}\}\) is nonempty.  
    } of $\dom(\pertfunc)$. \label{hyp:0-in-intdom}
    \item \(0\in\dom(\pertfunc)\) and \(\pertfunc(0) = 0\).\label{hyp:zero-minimized}
\end{enumerate}
Hypotheses \ref{hyp:lossfunc}-\ref{hyp:zero-minimized} are verified by many functions encountered in standard machine-learning problems. 
For example, this includes least-squares, logistic or hinge losses \cite{wang2020comprehensive} and terms $\pertfunc(\cdot)$ constructed as mixed-norms \cite{dedieu2021learning} or as the logarithm of Bayesian priors \cite{polson2019bayesian}. 

Finally, the complexity analysis of our method (see \Cref{sec:screening:complexity}) is discussed in view of the following assumption which holds for a wide range of problem instances encountered in practice, see \eg, Section~4.4.16 in \citep{beck2017first}: 
\begin{enumerate}[topsep=2pt,itemsep=3pt,parsep=0pt, label=({\(\mathrm{H}_\arabic*\)}),leftmargin=*]
    \setcounter{enumi}{4}
    \item The evaluation complexity of the convex conjugate of \(\lossfunc(\cdot)\) scales as \(\bigO(\ddim)\).
    \label{hyp:cost-conj}
\end{enumerate}

\subsection{Outline}

The rest of the paper is organized as follows. 
In \Cref{sec:bnb}, we introduce the main ingredients of standard \gls{bnb} algorithms.  
In \Cref{sec:screening}, we then present our new pruning strategy and discuss its impact on the \gls{bnb} algorithm.
Finally, our method is assessed numerically in \Cref{sec:numerics}.
To ease our exposition, all the proofs are deferred to \Cref{sec:proofs}. 

\subsection{Notational Conventions}
\label{sec:intro:notations}

Classical letters (\eg, $x$), boldface lowercase letters (\eg, $\pv$) and boldface uppercase letters (\eg, $\dic$) represent scalars, vectors and matrices, respectively.
$\0$ and $\1$ denote the all-zero and all-one vectors whose dimension is usually clear from the context. 
Vectorial operations involving equalities or inequalities have to be understood coordinate-wise.
We note $\pvi{\idxentry}$  the $\idxentry$-th entry of a vector $\pv$ and $\pv_{\setidx}$ its restriction to the entries indexed by some set of indices $\setidx$.
Similarly, we note $\atom{\idxentry}$ the $\idxentry$-th column of a matrix $\dic$ and $\dic_{\setidx}$ its restriction to the columns indexed by $\setidx$. 
$\intervint{a}{b}$ corresponds to the set of integers ranging from $a$ to $b$. 
The notations $\card{\cdot}$ and $\cdot\setminus\cdot$ are used to denote respectively the cardinality of a set and the difference between two sets.
$\icvx(\cdot)$ stands for the convex indicator function defined as $\icvx(\cdot) = 0$ if the condition in the parentheses is fulfilled and $\icvx(\cdot) = +\infty$ otherwise.
We let $\pospart{\pvi{}} = \max(\pvi{},0)$.
Given some proper function $\genericfunc(\cdot)$, we note $\dom(\genericfunc)$ its domain, $\conj{\genericfunc}(\cdot)$ its convex conjugate, $\biconj{\genericfunc}(\cdot)$ its convex  biconjugate 
and $\subdiff\genericfunc(\cdot)$ its subdifferential.
We refer to \cite{beck2017first} for a precise definition of these notions. 
Finally, we employ the notational convention $\genericfunc(\pv) = \sum_{\idxentry}\separable{\genericfunc}{\idxentry}(\pvi{\idxentry})$ when 
$\genericfunc(\cdot)$ is separable.


\section{Branch-and-Bound Algorithms}
\label{sec:bnb}

In this section, we outline the main ingredients of \gls{bnb} procedures.
We focus on the elements necessary to present our contribution and refer the reader to Chap.~5 in \cite{locatelli2013global} for a thorough description.

\begin{figure*}[t!]
    \centering
\begin{tikzpicture}

    \begin{scope}[xshift=-10cm]
        \node at (0,0) (node0) {};
        \draw[
            ultra thick,
            top color = white,
            bottom color = blue!30,
        ] (node0) circle (10pt) node {$\nu_0$};
        \node at ($(node0.south) + (0,-0.5)$) {\small{Solve $(\rpb^{\nodeSymb_0})$}};
        \node at ($(node0.south) + (0,-0.5)$) {}; 
    \end{scope}

    \begin{scope}[xshift=-6.75cm]
        \node at (0,0) (node0) {};
        \draw[->,dashed,ultra thick] ($(node0)+(-2.5,0)$) -- ($(node0)+(-1,0)$) node[midway,above,font=\scriptsize] {}; 
        \draw[
            ultra thick,
            top color = white,
            bottom color = blue!30,
        ] (node0) circle (10pt) node {$\nu_0$};
        \node at ($(node0)+(-1.3,-1.4)$) (node1) {};
        \draw[
            ultra thick,
            top color = white,
            bottom color = blue!30,
        ] (node1) circle (10pt) node {$\nu_1$};
        \draw[ultra thick,->] ($(node0.south west)+(-0.2,0)$) -- ($(node1.north)+(0.2,0.2)$) node[midway,fill=white,draw,font=\scriptsize,inner sep=2] {$\idxentry_0 \rightarrow \setzero$};
        \node at ($(node1.south) + (0,-0.5)$) {\small{Solve $(\rpb^{\nodeSymb_1})$}};
        \node at ($(node0)+(1.3,-1.4)$) (node2) {};
        \draw[
            ultra thick,
            top color = white,
            bottom color = blue!30,
        ] (node2) circle (10pt) node {$\nu_2$};
        \draw[ultra thick,->] ($(node0.south east)+(0.2,0)$) -- ($(node2.north)+(-0.2,0.2)$) node[midway,fill=white,draw,font=\scriptsize,inner sep=2] {$\idxentry_0 \rightarrow \setone$};
        \node at ($(node2.south) + (0,-0.5)$) {\small{Solve $(\rpb^{\nodeSymb_2})$}};
    \end{scope}

    \begin{scope}[xshift=-1cm]
        \node at (0,0) (node0) {};
        \draw[->,dashed,ultra thick] ($(node0)+(-4.25,0)$) -- ($(node0)+(-1.25,0)$) node[midway,font=\scriptsize,align=center] {$\robj^{\nodeSymb_1} \leq \UB{\pobj}$ \\~\\ $\robj^{\nodeSymb_2} > \UB{\pobj}$};
        \draw[
            ultra thick,
            top color = white,
            bottom color = blue!30,
        ] (node0) circle (10pt) node {$\nu_0$};
        \node at ($(node0)+(-1.3,-1.4)$) (node1) {};
        \draw[
            ultra thick,
            top color = white,
            bottom color = blue!30,
        ] (node1) circle (10pt) node {$\nu_1$};
        \draw[ultra thick,->] ($(node0.south west)+(-0.2,0)$) -- ($(node1.north)+(0.2,0.2)$);
        \node at ($(node0)+(1.3,-1.4)$) (node2) {};
        \draw[
            ultra thick,
            top color = white,
            bottom color = red!30,
        ] (node2) circle (10pt) node {$\nu_2$};
        \draw[ultra thick,->] ($(node0.south east)+(0.2,0)$) -- ($(node2.north)+(-0.2,0.2)$);
        \node at ($(node1)+(-1.3,-1.4)$) (node3) {};
        \draw[
            ultra thick,
            top color = white,
            bottom color = blue!30,
        ] (node3) circle (10pt) node {$\nu_3$};
        \draw[ultra thick,->] ($(node1.south west)+(-0.2,0)$) -- ($(node3.north)+(0.2,0.2)$) node[midway,fill=white,draw,font=\scriptsize,inner sep=2] {$\idxentry_1 \rightarrow \setzero$};
        \node at ($(node3.south) + (0,-0.5)$) {\small{Solve $(\rpb^{\nodeSymb_3})$}};
        \node at ($(node1)+(1.3,-1.4)$) (node4) {};
        \draw[
            ultra thick,
            top color = white,
            bottom color = blue!30,
        ] (node4) circle (10pt) node {$\nu_4$};
        \draw[ultra thick,->] ($(node1.south east)+(0.2,0)$) -- ($(node4.north)+(-0.2,0.2)$) node[midway,fill=white,draw,font=\scriptsize,inner sep=2] {$\idxentry_1 \rightarrow \setone$};
        \node at ($(node4.south) + (0,-0.5)$) {\small{Solve $(\rpb^{\nodeSymb_4})$}};
    \end{scope}
    
    \begin{scope}[xshift=4.5cm]
        \node at (0,0) (node0) {};
        \draw[->,dashed,ultra thick] ($(node0)+(-4,0)$) -- ($(node0)+(-1,0)$) node[midway,font=\scriptsize,align=center] {$\robj^{\nodeSymb_3} \leq \UB{\pobj}$ \\~\\ $\robj^{\nodeSymb_4} > \UB{\pobj}$};
        \draw[
            ultra thick,
            top color = white,
            bottom color = blue!30,
        ] (node0) circle (10pt) node {$\nu_0$};
        \node at ($(node0)+(-1.3,-1.4)$) (node1) {};
        \draw[
            ultra thick,
            top color = white,
            bottom color = blue!30,
        ] (node1) circle (10pt) node {$\nu_1$};
        \draw[ultra thick,->] ($(node0.south west)+(-0.2,0)$) -- ($(node1.north)+(0.2,0.2)$);
        \node at ($(node0)+(1.3,-1.4)$) (node2) {};
        \draw[
            ultra thick,
            top color = white,
            bottom color = red!30,
        ] (node2) circle (10pt) node {$\nu_2$};
        \draw[ultra thick,->] ($(node0.south east)+(0.2,0)$) -- ($(node2.north)+(-0.2,0.2)$);
        \node at ($(node1)+(-1.3,-1.4)$) (node3) {};
        \draw[
            ultra thick,
            top color = white,
            bottom color = blue!30,
        ] (node3) circle (10pt) node {$\nu_3$};
        \draw[ultra thick,->] ($(node1.south west)+(-0.2,0)$) -- ($(node3.north)+(0.2,0.2)$);
        \node at ($(node1)+(1.3,-1.4)$) (node4) {};
        \draw[
            ultra thick,
            top color = white,
            bottom color = red!30,
        ] (node4) circle (10pt) node {$\nu_4$};
        \draw[ultra thick,->] ($(node1.south east)+(0.2,0)$) -- ($(node4.north)+(-0.2,0.2)$);
        \draw[ultra thick,->,dashed] ($(node3.south east)+(0.2,0)$) -- ($(node3.south east)+(0.4,-0.3)$);
        \draw[ultra thick,->,dashed] ($(node3.south west)+(-0.2,0)$) -- ($(node3.south west)+(-0.4,-0.3)$);
    \end{scope}
\end{tikzpicture}
    \caption{
        Illustration of the \gls{bnb} decision-tree exploration. We note that a relaxation has to be solved at each node of the tree\protect\footnotemark{} to evaluate the lower bound~\eqref{eq:std-lb} involved in pruning test~\eqref{eq:pruning-test}. 
        Here, the pruning test is passed for nodes $\nodeSymb_2$ and $\nodeSymb_4$.
    }
    \label{fig:bnb}
\end{figure*}
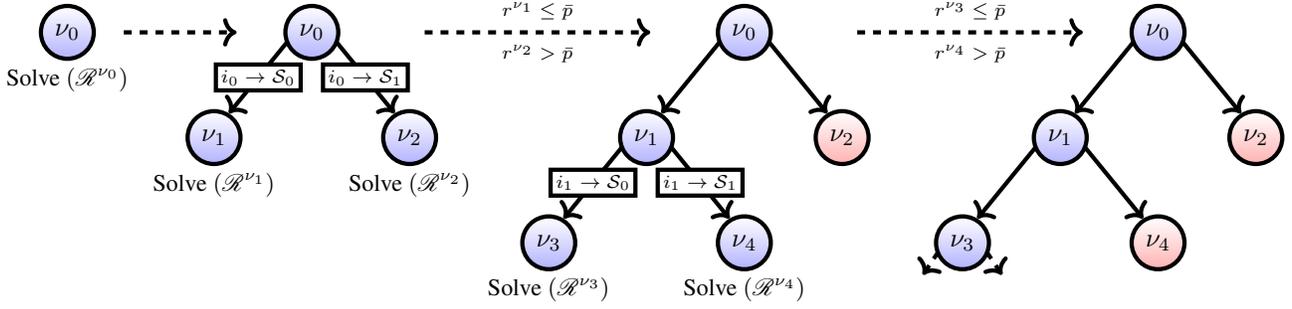

\subsection{Constructing and Pruning Regions}
\label{sec:bnb:pruning}

\gls{bnb} algorithms partition the feasible space of an optimization problem into regions and try to detect those that do not contain any optimizer.
In the specific context of \eqref{prob:prob}, standard \gls{bnb} implementations consider regions of the form
\begin{equation} \label{eq:region}
    \node{\pset} = \kset{\pv \in \kR^{\pdim}}{\pv_{{\setzero}} = \0, \pv_{{\setone}} \neq \0, \pv_{{\setnone}} \in \kR^{\card{{\setnone}}}}
\end{equation} 
where $\nodeSymb=(\setzero, \setone, \setnone)$ is a partition of $\intervint{1}{\pdim}$, see \cite{mhenni2020sparse,hazimeh2022sparse}.
Letting 
\begin{equation}
    \stepcounter{equation}
    \tag{$\node{\ppb}$}
    \label{prob:node-prob}
    \textstyle
    \node{\pobj} 
    = \inf_{\pv \in \kR^{\pdim}} \lossfunc(\dic\pv) + \regfunc(\pv) + \icvx(\pv \in \node{\pset})
    ,
\end{equation}
one can therefore deduce that no minimizer of \eqref{prob:prob} is contained in region $\node{\pset}$ when the inequality
\begin{equation}
    \label{eq:ideal-pruning-test}
    \node{\pobj} > \opt{\pobj}
\end{equation}
is verified.
Unfortunately, condition \eqref{eq:ideal-pruning-test} is of little interest in practice since evaluating $\node{\pobj}$ and $\opt{\pobj}$ is an NP-hard task. 
A workaround to this issue consists in relaxing \eqref{eq:ideal-pruning-test} as
\begin{equation}
    \label{eq:pruning-test}
    \node{\LB{\pobj}} > \UB{\pobj}
\end{equation}
where $\node{\LB{\pobj}}$ and $\UB{\pobj}$ are some \emph{tractable} lower and upper bounds on $\node{\pobj}$ and $\opt{\pobj}$, respectively. 
Inequality \eqref{eq:pruning-test} is often referred to as ``pruning test'' since if it is verified, $\node{\pset}$ does provably not contain any minimizer of \eqref{prob:prob} and can therefore be safely pruned from the optimization problem.

\footnotetext{
    From the point of view of pruning efficiency, solving~\eqref{prob:relax-node} at the root node \(\nodeSymb_0\) is unnecessary as the pruning test~\eqref{eq:pruning-test} is never passed.
    It is nevertheless common practice to solve it as a wide range of branching rules used in practice leverage its solutions, (see \textit{e.g.}, \Cref{sec:supp_numerics:implementation_choices}).
}

\subsection{Standard Bounding Strategy}
\label{sec:bnb:bounding}

The standard strategy to construct the bounds involved in \eqref{eq:pruning-test} is as follows.
First, an upper bound $\UB{\pobj}$ on $\opt{\pobj}$ can be computed by evaluating the objective function of \eqref{prob:prob} at any feasible point.\footnote{Many methods allow to construct relevant candidates at a reasonable cost, see \eg, \cite{wolsey1980heuristic}.}
Second, the computation of $\node{\LB{\pobj}}$ is generally done by minimizing some convex lower bound on the objective function of \eqref{prob:node-prob}. 
More specifically, a standard choice consists in replacing the term
\begin{equation}
    \label{eq:node-regfunc}
    \node{\regfunc}(\pv) = \regfunc(\pv) + \icvx(\pv \in \node{\pset})
\end{equation}
in problem \eqref{prob:node-prob} by its convex biconjugate denoted $\biconj{(\node{\regfunc})}(\cdot)$ in this paper, see Item~(i) of Proposition 13.16 in~\cite{bauschke2017convex}.
Hence, a valid choice for the lower bound in pruning test \eqref{eq:pruning-test} reads
\begin{equation}
    \label{eq:std-lb}
    \node{\LB{\pobj}} = \node{\robj}  
\end{equation}
where 
\begin{equation} \label{prob:relax-node}
    \stepcounter{equation}
    \tag{$\node{\rpb}$}
    \node{\robj} = \textstyle\inf_{\pv \in \kR^{\pdim}} \lossfunc(\dic\pv) + 
    \biconj{(\node{\regfunc})}(\pv).
\end{equation}
Problem \eqref{prob:relax-node} is called a ``relaxation'' of \eqref{prob:node-prob} and is usually addressed by first-order convex optimization methods~\cite{beck2017first}.
The complexity of these algorithms typically scales as $\bigO(\ddim\pdim \kappa)$ where $\kappa$ denotes the number of iterations performed by the numerical procedure. 

\subsection{Feasible Space Exploration}
\label{sec:bnb:exploration}

In \gls{bnb} procedures, the partitioning of the feasible set into regions can be identified with the expansion of a decision tree where each node corresponds to some region \(\node{\pset}\) defined as in~\eqref{eq:region}. 
As illustrated in \Cref{fig:bnb}, the exploration starts at the root node  \(\nodeSymb_0 = (\emptyset,\emptyset,\intervint{1}{\pdim})\) which corresponds to $\pset^{\nodeSymb_0}=\kR^\pdim$. 
For each leaf node $\nodeSymb=(\setzero,\setone,\setnone)$ examined by the \gls{bnb} procedure, problem \eqref{prob:relax-node} is solved and pruning test \eqref{eq:pruning-test} is evaluated using lower bound~\eqref{eq:std-lb}.
If the test is passed, the corresponding region $\pset^{\nodeSymb}$ is pruned from the problem and the exploration of the tree is stopped below this node.
If the pruning test is not passed, $\pset^{\nodeSymb}$ is partitioned into two new regions 
as follows. 
An index \(\idxentry\in\setnone\) is selected and the following two child nodes of $\nodeSymb$ are created:
\begin{subequations}
    \begin{align}
        \label{eq:direct-zero}
        \succnode{0}{\idxentry}&\triangleq ({\setzero}\cup\{\idxentry\},{\setone},{\setnone}\backslash\{\idxentry\})\\
        \label{eq:direct-one}
        \succnode{1}{\idxentry}&\triangleq({\setzero},{\setone}\cup\{\idxentry\},{\setnone}\backslash\{\idxentry\})
        .
    \end{align}
\end{subequations}
This process is repeated until all the leaf nodes of the tree are such that $\setnone=\emptyset$.
In the latter case,~\eqref{prob:node-prob} reduces to a convex optimization problem which can be commonly solved to machine precision.

\subsection{Pruning Effectiveness and Method Complexity}
\label{sec:bnb:complexity}

The computational burden associated to some node $\nodeSymb$ is generally dominated by the resolution of \eqref{prob:relax-node}. 
Since such a relaxation has to be solved at each node of the decision tree, the overall complexity of \gls{bnb} procedures typically roughly scales linearly with the total number of nodes explored by the algorithm.
On the one hand, we note that passing pruning test \eqref{eq:pruning-test} often requires to compute \emph{tight} lower bounds on \(\node{\pobj}\), in the sense that \(\node{\pobj}- \node{\robj}\) should be small.
On the other hand, achieving the prescribed tightness usually imposes exploring \emph{deep} nodes in the decision tree.\footnote{This claim is discussed more precisely in \Cref{sec:proofs:complexity}.}
As a consequence, the number of nodes explored by \gls{bnb} procedures may become prohibitively large and thus lead to impractical solving times.  
In the next section, we present a strategy to alleviate this computational bottleneck.


\section{A New Pruning Strategy}
\label{sec:screening}

In this section, we present our proposed new strategy to implement pruning test \eqref{eq:pruning-test}. 
Our exposition is organized as follows.
In \Cref{sec:screening:bound}, we first discuss another lower bound on \(\node{\pobj}\) and show that the latter can be evaluated at low cost for a \emph{set} of nodes in \Cref{sec:screening:complexity}.~Using this observation, we explain in \Cref{sec:screening:impact} how the \gls{bnb} decision tree can be expanded when several pruning tests exploiting this lower bound are passed simultaneously.~Finally, in \Cref{sec:screening:implementation} we describe how to embed the proposed methodology in standard \gls{bnb} implementations at virtually no cost.

\subsection{Exploiting Duality to Construct Lower Bounds}
\label{sec:screening:bound}

We propose to construct a lower bound on $\node{\pobj}$ that does not require to solve \eqref{prob:relax-node}.~To this end, we follow another line of work in the \gls{bnb} literature~\citep{sarin1988surrogate} and leverage the Fenchel-Rockafellar dual problem \cite{rockafellar1967duality} associated to~\eqref{prob:relax-node}. 
Under hypotheses \ref{hyp:lossfunc}-\ref{hyp:pertfunc}, the latter is given by\footnote{
    Here, we use the fact that \(\conj{(\biconj{(\node{\regfunc})})}(\cdot) = \conj{(\node{\regfunc})}(\cdot)\) under~\ref{hyp:pertfunc} by Proposition 13.16.
}
\begin{equation}
    \stepcounter{equation}
    \tag{$\node{\mathscr{D}}$}
    \label{prob:dual-node}
    \textstyle
    \node{\dobj} = \sup_{\dv \in \kR^{\ddim}} \underbrace{-\conj{\lossfunc}(-\dv) - \conj{(\node{\regfunc})}(\transp{\dic}\dv)}_{\triangleq~\node{\dfunc}(\dv)}
\end{equation}
and verifies the inequality
\begin{equation}
    \label{eq:bound-ordering}
    \node{\dobj} \leq \node{\robj},
\end{equation}
which is tight under mild assumptions, see \eg, Proposition~15.22 in \cite{bauschke2017convex}.
Hence, setting
\begin{equation}
    \label{eq:dual-bound}
    \node{\LB{\pobj}} = \node{\dfunc}(\dv)
\end{equation}
leads to a valid lower bound to implement~\eqref{eq:pruning-test} for any \(\dv \in \kR^{\ddim}\).
Regarding inequality \eqref{eq:bound-ordering}, we note that lower bound~\eqref{eq:dual-bound} may not be as tight as the standard one given in~\eqref{eq:std-lb}. 
Nevertheless, we show in the sequel that \eqref{eq:dual-bound} can be evaluated for several ``successors'' of node $\nodeSymb$ at virtually no cost.

Let us first precise the notion of ``successors'' of a node \(\nodeSymb\): 
\begin{definition}
    \label{def:successor}
    The node $\nodeSymb'=(\setzero',\setone',\setnone')$ is said to be a successor of $\nodeSymb=(\setzero,\setone,\setnone)$ if it verifies
    \begin{equation}
        \label{eq:successor}
        \setzero \subseteq \setzero'
        \quad\text{and}\quad
        \setone \subseteq \setone'
        .
    \end{equation}
    Moreover, $\nodeSymb'$ is said to be a direct successor of $\nodeSymb$ if it fulfills property~\eqref{eq:successor} and verifies
    \begin{equation}
        \label{eq:direct-successor}
        (\setzero' \setminus \setzero) \cup (\setone' \setminus \setone) = \{\idxentry\}
    \end{equation}
    for some \(\idxentry \in \setnone\).
\end{definition}
From a \gls{bnb} tree perspective, direct successors correspond to the nodes \(\succnode{0}{\idxentry}\) and \(\succnode{1}{\idxentry}\) described in~\eqref{eq:direct-zero}-\eqref{eq:direct-one} for some \(\idxentry\in\setnone\).

Interestingly, the objective functions of the dual problems~\eqref{prob:dual-node} at some node $\nodeSymb$ and any of its successors share a similar mathematical structure.
To reveal this link, we first establish the following result in \Cref{sec:proof:dual-regfunc}:  
\begin{proposition}
    \label{prop:dual-relaxregfunc}
    Let $\nodeSymb=(\setzero,\setone,\setnone)$ be a node.
    Under \ref{hyp:pertfunc}-\ref{hyp:zero-minimized}, the function 
    $\conj{(\node{\regfunc})}(\cdot)$ is separable and defined coordinate-wise for all $\vvi{} \in \kR$ as
    \begin{equation}
        \label{eq:dual-relaxregfunc}
        \conj{(\separable{\node{\regfunc}}{\idxentry})}(\vvi{}) = 
        \begin{cases}
            0 &\text{if} \ \idxentry \in \setzero \\
            \conj{\pertfunc}(\vvi{}) - \reg &\text{if} \ \idxentry \in \setone \\
            \pospart{\conj{\pertfunc}(\vvi{}) - \reg} &\text{if} \ \idxentry \in \setnone.
        \end{cases} 
    \end{equation}
\end{proposition}
As a consequence of \Cref{prop:dual-relaxregfunc}, we prove in \Cref{sec:proof:dual-link} that the objective function of the dual problem \eqref{prob:dual-node} at some node $\nodeSymb$ verifies a notable relation with that of its successors.
More precisely, letting
\begin{subequations}
    \begin{align}
        \label{eq:pivot-zero}
        \pivot{0}(\vvi{}) &\triangleq \pospart{\conj{\pertfunc}(\vvi{}) - \reg} \\
        \label{eq:pivot-one}
        \pivot{1}(\vvi{}) &\triangleq \pospart{\reg - \conj{\pertfunc}(\vvi{})}
    \end{align}
\end{subequations}
for all \(\vvi{} \in \kR\), we obtain the following property: 
\begin{proposition}
    \label{prop:dual-link}
    Let $\nodeSymb' = (\setzero',\setone',\setnone')$ be a successor of $\nodeSymb = (\setzero,\setone,\setnone)$. 
    Under~\ref{hyp:lossfunc}-\ref{hyp:zero-minimized}, we have for all \(\dv \in \kR^{\ddim}\):
    \begin{equation} \label{eq:dual-function}
        \begin{split}
            \dfunc^{\nodeSymb'}(\dv) = \node{\dfunc}(\dv) &+ \textstyle\sum_{\idxentry \in \setzero'\setminus\setzero} \pivot{0}(\transp{\atom{\idxentry}}\dv) \\
            &+ \textstyle\sum_{\idxentry \in \setone'\setminus\setone} \pivot{1}(\transp{\atom{\idxentry}}\dv) 
            .
        \end{split}
    \end{equation}
\end{proposition}
From \Cref{prop:dual-link}, we note that the term $\node{\dfunc}(\dv)$ is common to the objective functions of all the dual problems associated to the successors of $\nodeSymb$.
As a consequence, lower bound~\eqref{eq:dual-bound} can be computed for \emph{several} successors $\nodeSymb'$ of $\nodeSymb$ through a \emph{single} evaluation of this term. 
In the next section, we leverage this observation to \emph{jointly} evaluate lower bound~\eqref{eq:dual-bound} at all the direct successors of \(\nodeSymb\) with a complexity scaling as \(\bigO(\ddim\pdim)\).

\subsection{Evaluation of \texorpdfstring{\eqref{eq:dual-bound}}{(~\ref{eq:dual-bound})} for All the Direct Successors} \label{sec:screening:complexity}

In this section, we consider the scenario where one wants to evaluate lower bound~\eqref{eq:dual-bound} for all the direct successors of some node $\nodeSymb=(\setzero,\setone,\setnone)$. 
In this case, letting \(\nodeSymb'=(\setzero',\setone',\setnone')\) denote a direct successor of \(\nodeSymb\), we have that \eqref{eq:dual-function} simplifies to
\begin{equation} \label{eq:dual-function-direct}
    \hspace*{-3pt}
    \dfunc^{\nodeSymb'}(\dv) = 
    \node{\dfunc}(\dv) +
    \begin{cases}
        \pivot{0}(\transp{\atom{\idxentry}}\dv) & \hspace*{-5pt}\mbox{if \(\setzero' \setminus \setzero=\{\idxentry\}\)} 
        \\
        \pivot{1}(\transp{\atom{\idxentry}}\dv) & \hspace*{-5pt}\mbox{if \(\setone' \setminus \setone=\{\idxentry\}\)}.
    \end{cases}
\end{equation}
We thus remark that the evaluation of~\eqref{eq:dual-bound} for all the direct successors $\nodeSymb'$ of $\nodeSymb$ amounts to computing:
\begin{enumerate}[nosep, label=\textit{\roman*)}]
    \item \label{task:computing inner products}
    the inner products \(\{\transp{\atom{\idxentry}}{\dv}\}_{\idxentry=1}^{\pdim}\),
    \item the quantity $\node{\dfunc}(\dv)$,
    \item the terms $\pivot{0}(\transp{\atom{\idxentry}}\dv)$ and $\pivot{1}(\transp{\atom{\idxentry}}\dv)$ for all $\idxentry \in \setnone$.
\end{enumerate}
The first task can obviously be done with a complexity $\bigO(\ddim\pdim)$. 
Moreover, using the definition of $\node{\dfunc}(\dv)$ in \eqref{prob:dual-node} and \Cref{prop:dual-relaxregfunc}, 
we have that the last two tasks can be implemented with a complexity scaling as \(\bigO(\ddim+\pdim)\) and \(\bigO(\card{\setnone})\), respectively, under hypothesis~\ref{hyp:cost-conj}.\footnote{This is achieved by re-using the computations of task $i)$.}
Overall, the computational burden induced by the evaluation of~\eqref{eq:dual-bound} for \emph{all} the direct successors of node $\nodeSymb$ scales as $\bigO(\ddim\pdim)$.
This is a substantial decrease with respect to the complexity required to perform the same task for standard lower bound \eqref{eq:std-lb}. 
More specifically, applying $\kappa$ iterations of some first-order method to \eqref{prob:relax-node} leads to a complexity of $\mathcal{O}(\ddim\pdim\kappa)$ (see \Cref{sec:bnb:bounding}). Repeating this operation for all $2|\setnone|$ direct successors of $\nodeSymb$ then leads to a degradation of the computation burden by a factor $2|\setnone|\kappa$.  

In fact, we will show in \Cref{sec:screening:implementation} that the complexity \textit{overhead} needed to evaluate lower bound~\eqref{eq:dual-bound} for all the direct successors of $\nodeSymb$ can be further reduced from $\bigO(\ddim\pdim)$ to $\bigO(\ddim+\pdim)$ within a standard \gls{bnb} implementation.

\subsection{Simultaneous Tests: Effect on the Tree Exploration}
\label{sec:screening:impact}

The evaluation at low cost of lower bound~\eqref{eq:dual-bound} for all the direct successors $\nodeSymb'$ of \(\nodeSymb\) opens the way to the practical implementation of \emph{simultaneous} pruning tests for several regions $\pset^{\nodeSymb'}$. 
In this section, we investigate how the \gls{bnb} decision tree can be expanded if pruning test~\eqref{eq:pruning-test} is passed simultaneously for \textit{several} direct successors of some node \(\nodeSymb=(\setzero,\setone,\setnone)\). 

We first note that
\(
\kset{\succnode{0}{\idxentry}}{\idxentry\in\setnone}
\cup
\kset{\succnode{1}{\idxentry}}{\idxentry\in\setnone}
\), where \(\succnode{0}{\idxentry}\) and \(\succnode{1}{\idxentry}\) have been defined in~\eqref{eq:direct-zero}-\eqref{eq:direct-one}, 
corresponds to all the direct successors of \(\nodeSymb\). 
Hence, the two sets 
\begin{subequations}
    \begin{align}
        \node{\setidxzero} &= \kset{\idxentry\in\setnone}{\dfunc^{\succnode{0}{\idxentry}}(\dv) > \UB{\pobj}}\label{def:I0}\\
        \node{\setidxone} &= \kset{\idxentry\in\setnone}{\dfunc^{\succnode{1}{\idxentry}}(\dv) > \UB{\pobj}}\label{def:I1}
    \end{align}
\end{subequations}
characterize the direct successors of $\nodeSymb$ satisfying pruning condition \eqref{eq:pruning-test} when implemented with the proposed lower bound~\eqref{eq:dual-bound} at some dual point $\dv\in\kR^\ddim$.\footnote{
As discussed in \Cref{sec:screening:implementation}, the dual point $\dv$ is constructed from the iterates of the solving procedure addressing \eqref{prob:relax-node}.}
By definition of the sets $\node{\setidxzero}$ and $\node{\setidxone}$, the region
\begin{equation}
    (\cup_{\idxentry\in\node{\setidxzero}} \pset^{\succnode{0}{\idxentry}})
    \cup
    (\cup_{\idxentry\in\node{\setidxone}} \pset^{\succnode{1}{\idxentry}})
\end{equation}
does not contain any minimizer of problem \eqref{prob:prob}.

In the rest of this paragraph, we show how the simultaneous success of several pruning tests can be translated in terms of deployment of the \gls{bnb} decision tree.
To guide our reasoning, we examine two cases and then present a generic procedure to expand the \gls{bnb} decision tree in \Cref{algo:expansion-tree}.

First, suppose that $\node{\setidxzero}\cap\node{\setidxone} \neq \emptyset$, that is there exists some $\idxentry\in{\setnone}$ such that both $\succnode{0}{\idxentry}$ and $\succnode{1}{\idxentry}$ pass the pruning test~\eqref{eq:pruning-test} with lower bound \eqref{eq:dual-bound}.
Since $\node{\pset}=\pset^{\succnode{1}{\idxentry}}\cup \pset^{\succnode{0}{\idxentry}}$, one concludes that no minimizers to problem~\eqref{prob:prob} can be found in region \(\node{\pset}\) which can thus be pruned from the decision tree. \\
Second, suppose that \(\node{\setidxzero}\cap\node{\setidxone} = \emptyset\) but \(\node{\setidxzero}\cup\node{\setidxone} \neq \emptyset\). 
We first examine the case where only one direct successor of \(\nodeSymb\) passes the pruning test~\eqref{eq:pruning-test} with lower bound~\eqref{eq:dual-bound}, \textit{i.e.}, $\node{\setidxzero}\cup\node{\setidxone}=\{\idxentry\}$, say $\node{\setidxzero}=\{\idxentry\}$ and \(\node{\setidxone}=\emptyset\) for instance.
Then, one concludes that region $\pset^{\succnode{0}{\idxentry}}$ does not contain any solution to problem \eqref{prob:prob} and can therefore be pruned from the problem's feasible set without altering its solution. 
From a decision-tree perspective, this information can be taken into account by: \textit{i)} creating two new nodes $\succnode{0}{\idxentry}$ and $\succnode{1}{\idxentry}$ below $\nodeSymb$ and \textit{ii)} immediately pruning the region \(\pset^{\succnode{0}{\idxentry}}\).
If \(\node{\setidxzero}\cup\node{\setidxone}\) contains more than one element, the next proposition suggests that this procedure can be applied recursively.

\begin{proposition}
    \label{prop:propagation-tests}
    Under \ref{hyp:lossfunc}-\ref{hyp:zero-minimized}, we have for all successors \(\nodeSymb'=(\setzero',\setone',\setnone')\) of \(\nodeSymb\) and indices \(\idxentry\in\setnone'\):
    \begin{subequations}
        \begin{align}
            \dfunc^{\succnode{0}{\idxentry}}(\dv) > \UB{\pobj} 
            &\;\implies\; \dfunc^{\nodeSymb'_{0,\idxentry}}(\dv) > \UB{\pobj}\\
            \dfunc^{\succnode{1}{\idxentry}}(\dv) > \UB{\pobj} 
            &\;\implies\; \dfunc^{\nodeSymb'_{1,\idxentry}}(\dv) > \UB{\pobj}
            .
        \end{align}
    \end{subequations}
\end{proposition}    

A proof of this result can be found in \Cref{sec:proof:propagation-tests}. 
\Cref{prop:propagation-tests} states that if pruning test~\eqref{eq:pruning-test} is passed using lower bound~\eqref{eq:dual-bound}
for some direct successor of node $\nodeSymb$, the result of the test can be propagated to any successor $\nodeSymb'$ of $\nodeSymb$ compliant with the condition $\idxentry\in\setnone'$.  
This observation leads to \Cref{algo:expansion-tree} which describes how the \gls{bnb} decision tree can be expanded upon the knowledge of $\node{\setidxzero}$ and $\node{\setidxone}$.
\Cref{fig:bnb-scr} illustrates the output of \Cref{algo:expansion-tree} when $\nodeSymb=\nodeSymb_0$, $\setidxzero^{\nodeSymb_0}=\emptyset$ and $\setidxone^{\nodeSymb_0}=\{\idxentry_0,\idxentry_1\}$.
In comparison to \Cref{fig:bnb}, we observe that the proposed pruning procedure does not require to solve the relaxations at the nodes $\nodeSymb_1$, $\nodeSymb_2$ and $\nodeSymb_4$, although ultimately leading to the same expanded tree.

\begin{algorithm}[t]
    \small 
    \SetKwInOut{Input}{input}
    \Input{node $\nodeSymb$, sets $\node{\setidxzero}$, $\node{\setidxone}$ defined in \eqref{def:I0}-\eqref{def:I1}}\vspace{5pt}

    \uIf{$\node{\setidxzero}\cap \node{\setidxone}\neq\emptyset$}{Prune \(\pset^{\nodeSymb}\) from the \gls{bnb} tree}
    \Else{
        Set $\nodeSymb' \leftarrow \nodeSymb$\;

        \ForAll{$\idxentry \in \node{\setidxzero} \cup \setidxone$}{
            Create the two direct successors $\nodeSymb'_{0,\idxentry}$ and $\nodeSymb'_{1,\idxentry}$ to $\nodeSymb'$\;

            \uIf{$i\in\node{\setidxzero}$}{
                Prune \(\pset^{\nodeSymb'_{0,\idxentry}}\) and set $\nodeSymb' \leftarrow \nodeSymb'_{1,\idxentry}$\;
            }\ElseIf{$i\in\node{\setidxone}$}{
                Prune \(\pset^{\nodeSymb'_{1,\idxentry}}\) and set $\nodeSymb' \leftarrow \nodeSymb'_{0,\idxentry}$\;
            }
        }
    }
    \caption{\small Tree expansion based on simultaneous pruning \label{algo:expansion-tree}}
\end{algorithm}

\subsection{Implementation in Branch-and-Bound Methods} \label{sec:screening:implementation}

In this final section, we discuss how the proposed pruning strategy, leveraging lower bound~\eqref{eq:dual-bound}, can be efficiently integrated into standard \gls{bnb} implementations. 

We consider the following strategy.  
Given some node $\nodeSymb$ processed by the \gls{bnb} procedure, we test simultaneously all the direct successors of $\nodeSymb$ at each iteration of the solving process of \eqref{prob:relax-node}.
More specifically, letting $\hat{\pv} \in \kR^{\pdim}$ be the current iterate constructed by the first-order method addressing~\eqref{prob:relax-node}, we evaluate the lower bound~\eqref{eq:dual-bound} with some dual point verifying
\begin{equation} \label{eq:dual-point}
    \dv \in -\subdiff\lossfunc(\dic\hat{\pv}) 
\end{equation}
and construct the sets $\node{\setidxzero}$ and $\node{\setidxone}$ according to \eqref{def:I0}-\eqref{def:I1}. 
We defer the discussion on our motivations for choosing \(\dv\) as in~\eqref{eq:dual-point} to the next paragraph. 
We stop (prematurely) the resolution of~\eqref{prob:relax-node} as soon as $\node{\setidxzero}\cup\node{\setidxone}\neq\emptyset$ for some $\dv \in \kR^{\ddim}$ and expand the decision tree as described in \Cref{algo:expansion-tree}. 
The \gls{bnb} algorithm is then continued and this process is repeated.
If none of the pruning tests are passed during the resolution of \eqref{prob:relax-node}, the standard pruning test using lower bound \eqref{eq:std-lb} is applied and the decision tree is expanded according to standard \gls{bnb} operations as described in \Cref{sec:bnb}.

We devote the rest of the section to motivate the choice of dual point \(\dv \in \kR^{\ddim}\) described in~\eqref{eq:dual-point}.
From an effectiveness point of view, our rationale is to (try to) maximize the first term in~\eqref{eq:dual-function-direct}. 
If strong duality holds between \eqref{prob:relax-node} and \eqref{prob:dual-node}, by virtue of Theorem 19.1 from \citep{bauschke2017convex}, this can be achieved by choosing \(\dv \in - \partial \lossfunc(\dic\opt{\pv})\) where \(\opt{\pv}\) denotes any minimizer of~\eqref{prob:relax-node}. 
Since such a minimizer is not available, we use the current iterate (denoted \(\hat{\pv}\) in \eqref{eq:dual-point}) of the numerical procedure solving~\eqref{prob:relax-node} as a surrogate.

From a complexity point of view, this proposed pruning methodology can be integrated within standard \gls{bnb} implementations at virtually no cost. 
Indeed, we notice that a dual point $\dv \in \kR^{\ddim}$ verifying~\eqref{eq:dual-point} and the corresponding vector $\transp{\dic}\dv \in \kR^{\pdim}$ are already computed during the iterations of many standard first-order methods tailored to solve~\eqref{prob:relax-node} such as proximal gradient, coordinate descent or ADMM~\cite{beck2017first}.
According to our discussion in \Cref{sec:screening:bound}, the complexity overhead required to compute lower bound~\eqref{eq:dual-bound} associated to \emph{all} the direct successors of $\nodeSymb$ then drops from $\bigO(\ddim\pdim)$ to $\mathcal{O}(\ddim+\pdim)$ since the inner products $\{\ktranspose{\atom{}}_i{\dv}\}_{i=1}^{\pdim}$ involved in task \ref{task:computing inner products} are already available. 
We note that this additional computational burden is negligible as compared to the complexity of standard first-order methods which typically scales as $\mathcal{O}(\ddim\pdim)$ per iteration.

\begin{figure}[t]
    \centering

\begin{tikzpicture}
    \begin{scope}[xshift=-4cm]
        \node at (0,0) (node0) {};
        \draw[
            ultra thick,
            top color = white,
            bottom color = blue!30,
        ] (node0) circle (10pt) node {$\nu_0$};
        \node at ($(node0.south) + (0,-0.5)$) {{\small{Solve $(\rpb^{\nodeSymb_0})$}}};
    \end{scope}
    
    \begin{scope}[xshift=1cm]
        \node at (0,0) (node0) {};
        \draw[->,dashed,ultra thick] ($(node0)+(-4,0)$) -- ($(node0)+(-1,0)$) node[midway,font=\scriptsize,align=left] {$\setidxzero^{\nodeSymb_0}=\emptyset$ \\~\\ $\setidxone^{\nodeSymb_0}=\{\idxentry_0,\idxentry_1\}$};
        \draw[
            ultra thick,
            top color = white,
            bottom color = blue!30,
        ] (node0) circle (10pt) node {$\nu_0$};
        \node at ($(node0)+(-1.3,-1.4)$) (node1) {};
        \draw[
            ultra thick,
            top color = white,
            bottom color = blue!30,
        ] (node1) circle (10pt) node {$\nu_1$};
        \draw[ultra thick,->] ($(node0.south west)+(-0.2,0)$) -- ($(node1.north)+(0.2,0.2)$) node[midway,fill=white,draw,font=\scriptsize,inner sep=2] {$\idxentry_0 \rightarrow \setzero$};
        \node at ($(node0)+(1.3,-1.4)$) (node2) {};
        \draw[
            ultra thick,
            top color = white,
            bottom color = red!30,
        ] (node2) circle (10pt) node {$\nu_2$};
        \draw[ultra thick,->] ($(node0.south east)+(0.2,0)$) -- ($(node2.north)+(-0.2,0.2)$);
        \node at ($(node2.south) + (0,-0.5)$) {\small$(\rpb^{\nodeSymb_2})$ not solved};
        \node at ($(node1)+(-1.2,-1.4)$) (node3) {};
        \draw[
            ultra thick,
            top color = white,
            bottom color = blue!30,
        ] (node3) circle (10pt) node {$\nu_3$};
        \draw[ultra thick,->] ($(node1.south west)+(-0.2,0)$) -- ($(node3.north)+(0.2,0.2)$) node[midway,fill=white,draw,font=\scriptsize,inner sep=2] {$\idxentry_1 \rightarrow \setzero$};
        \node at ($(node3.south) + (0,-0.5)$) {};
        \node at ($(node1.west) + (-1.5,0)$) {\small$(\rpb^{\nodeSymb_1})$ not solved};
        \node at ($(node1)+(1.2,-1.4)$) (node4) {};
        \draw[
            ultra thick,
            top color = white,
            bottom color = red!30,
        ] (node4) circle (10pt) node {$\nu_4$};
        \draw[ultra thick,->] ($(node1.south east)+(0.2,0)$) -- ($(node4.north)+(-0.2,0.2)$);
        \node at ($(node4.south) + (0,-0.5)$) {\small$(\rpb^{\nodeSymb_4})$ not solved};
    \end{scope}
\end{tikzpicture}
    \caption{
        Impact of simultaneous pruning tests on the \gls{bnb} tree exploration. Output of \Cref{algo:expansion-tree} when applied with $\nodeSymb=\nodeSymb_0$, $\setidxzero^{\nodeSymb_0}=\emptyset$ and $\setidxone^{\nodeSymb_0}=\{\idxentry_0,\idxentry_1\}$.
    }
    \label{fig:bnb-scr}
\end{figure}
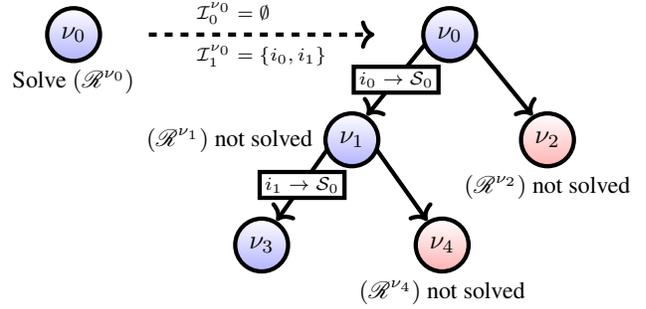


\section{Numerical Experiments}
\label{sec:numerics}

\newcommand{\elops}{\texttt{\texttt{El0ps}}}

In this final section, we assess numerically the proposed pruning strategy to accelerate \gls{bnb} algorithms addressing problem \eqref{prob:prob}.
We do not focus on the statistical characterization of the solutions but refer to \cite{bertsimas2020sparsereg,hastie2020best} for a thorough discussion on this topic.

\paragraph{Reproducibility}
The research presented in this paper is reproducible. 
The associated code is open-sourced\footnote{\url{https://github.com/TheoGuyard/El0ps}} and all the datasets used in our simulations are publicly available. 
Computations were carried out using the Grid'5000 testbed, supported by a scientific interest group hosted by INRIA and including CNRS, RENATER and several universities as well as other organizations.\footnote{\url{https://www.grid5000.fr}}
Experiments were run on a Debian 10 operating system, featuring one Intel Xeon E5-2660 v3 CPU clocked at 2.60 GHz with 16 GB of RAM.

\paragraph{Solver specifications}
In our comparisons, we consider different methods solving problem \eqref{prob:prob} exactly, that is returning the value of (at least) one minimizer to machine precision.
First, we use \texttt{Mosek}, \texttt{Cplex} and \texttt{Gurobi} which are off-the-shelf \gls{mip} solvers \cite{anand2017comparative}.
Second, we consider the \texttt{L0bnb} solver \cite{hazimeh2022sparse} which is dedicated to some specific\footnote{Namely with a quadratic function $\lossfunc(\cdot)$ and a term $\pertfunc(\cdot)$ corresponding to an $\ell_2$-norm and/or a bound constraint.} instances of problem~\eqref{prob:node-prob}.
We compare these procedures to a standard \gls{bnb} implementation enhanced with the simultaneous pruning tests described in this paper, noted \elops{}.

For the sake of reproducibility, the \gls{mip} formulations of the problem considered are specified in \Cref{sec:supp_numerics:mip}.
\Cref{sec:supp_numerics:implementation_choices} details our \gls{bnb} implementation choices and \Cref{sec:supp_numerics:implementation} gives the expression of the function $\biconj{(\node{\regfunc})}(\cdot)$ involved in the relaxations \eqref{prob:relax-node} for the instances of problem \eqref{prob:prob} considered in this section.
Finally, the tuning procedure for $\reg$ and the hyperparameters involved in $\pertfunc(\cdot)$ is described in \Cref{sec:supp_numerics:hyperparameters}.

\subsection{Performance on Synthetic Data}
\label{sec:numerics:synthetic}

In this section, we analyze the performance of different solvers on synthetic data. 
We consider instances of problem \eqref{prob:prob} defined by
\begin{subequations}
    \begin{align}
        \label{eq:leastsquares}
        \lossfunc(\cdot) &= \tfrac{1}{2}\norm{\obs-\cdot}{2}^2 \\
        \label{eq:bigm}
        \pertfunc(\cdot) &= \icvx(\abs{\cdot} \leq \bigM)
    \end{align}
\end{subequations}
for some $\obs \in \kR^{\ddim}$ and $\bigM > 0$.
This choice is motivated by various applications, see \eg, \cite{tillmann2021cardinality,bertsimas2021unified,bertsimas2023compressed}.
Our results are averaged over 100 instances independently generated.

\paragraph{Instance generation}

For each problem instance, we generate the rows of $\dic \in \kR^{\ddim\times\pdim}$ as independent realizations of a multivariate normal distribution with zero mean and covariance matrix $\corrmat \in \kR^{\pdim\times \pdim}$ where each entry $(i,j)$ is defined as $\corrmatel{ij}=\corrparam^{\abs{i-j}}$ for some $\corrparam \in [0,1)$. 
Moreover, we set $\obs = \dic\groundtruth + \noise$ where $\groundtruth \in \kR^{\pdim}$ has $\sparsitylevel$ evenly-spaced non-zero entries of unit amplitude in absolute value and where $\noise \in \kR^{\ddim}$ is a zero-mean Gaussian noise with a variance tuned to obtain some signal-to-noise ratio $\snr = 10 \log_{10}(\norm{\dic\groundtruth}{2}^2/\norm{\noise}{2}^2)$.

\paragraph{Performance profiles}
\label{sec:numerics:synthetic:perfprofile}

We generate each problem instance as described above with the parameters $\sparsitylevel=5$, $\ddim=500$, $\pdim=1000$, $\corrparam=0.9$ and $\snr=10$.
\Cref{fig:perfprofile} represents the percentage of instances solved (to machine precision) by each method within a given time budget.

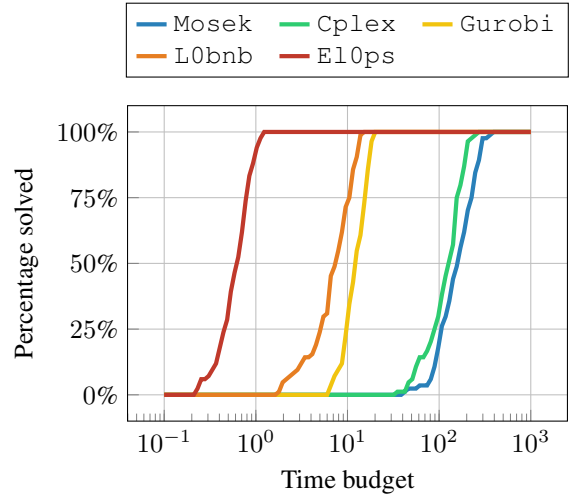
\begin{figure}[!t]
    \centering

\pgfplotscreateplotcyclelist{cycle_list_perfprofile}{
    mosekcolor, ultra thick\\    
    cplexcolor, ultra thick\\
    gurobicolor, ultra thick\\
    l0bnbcolor, ultra thick\\
    bnbcolor, ultra thick\\
}

\pgfplotsset{
 	legend image code/.code={
		\draw[mark repeat=2,mark phase=2] plot coordinates {
			(0cm,0cm)
			(0.4cm,0cm)
			(0.4cm,0cm)
		};
    }
}

\begin{tikzpicture}
    \begin{groupplot}[
        group style     = {
            group size      = 1 by 1,
            ylabels at      = edge left,
            yticklabels at  = edge left,
            horizontal sep  = 10pt,
        },
        height          = 0.7 \linewidth,
        width           = 0.9 \linewidth,
        xmode           = log,
        ytick pos       = left,
        xtick pos       = bottom,
        ytick           = {0,0.25,0.50,0.75,1},
        yticklabel      = {\pgfmathparse{100*\tick}\pgfmathprintnumber{\pgfmathresult}\%},
        grid            = major,
        cycle list name = cycle_list_perfprofile,
    ]

        \nextgroupplot[
            xlabel          = Time budget,
            ylabel          = Percentage solved,
            legend to name = perfprofile_legend,
            legend columns = 3,
            legend cell align={left},
            legend style={/tikz/every even column/.append style={column sep=0.25cm}},
        ]
        \foreach \solver in {mosek,cplex,gurobi,l0bnb,proposed}{
            \addplot table[
                x       = grid, 
                y       = \solver,
                col sep = comma,
            ] {data/perfprofile.csv};
        }
        \addlegendentry{\texttt{Mosek}}
        \addlegendentry{\texttt{Cplex}}
        \addlegendentry{\texttt{Gurobi}}
        \addlegendentry{\texttt{L0bnb}}
        \addlegendentry{\elops}
    \end{groupplot}
    \path (group c1r1.north east) -- node[above,yshift=0.25cm]{\ref{perfprofile_legend}} (group c1r1.north west);
\end{tikzpicture}
    \vspace*{-0.5cm}
    \caption{Performance profiles of different solvers.}
    \label{fig:perfprofile}
\end{figure}

We notice that regardless of the considered time budget, our method can solve a larger proportion of instances than its competitors.
In particular, all the problem instances are solved within a time budget for which no instances have been solved by the other methods.
More specifically, we observe that our methodology enables an acceleration of at least one order of magnitude with respect to the other procedures to solve all the problem instances.
In particular, we mention that the \gls{bnb} implementation choices of \elops{} are similar to those of \texttt{L0bnb}. 
This suggests that the observed improvement in terms of computation time is essentially due to our simultaneous pruning strategy.

\paragraph{Sensibility study}
To study more finely the gains permitted by the contribution proposed in this paper, we compare two different versions of a \gls{bnb} algorithm.
The first one implements the standard pruning strategy with the lower bound~\eqref{eq:std-lb} whereas the second one implements both this standard strategy and the proposed simultaneous pruning methodology involving lower bound~\eqref{eq:dual-bound}. 
We generate synthetic problem instances as described above by varying one parameter at a time to cover different working regimes.
In \Cref{fig:sensibility}, we represent the acceleration factor in terms of solving time obtained by our novel pruning strategy. 
More precisely, it is defined as the ratio of the solving times obtained with \elops{} where the simultaneous pruning tests (see \Cref{sec:screening:impact}) have been disabled and \elops{}.

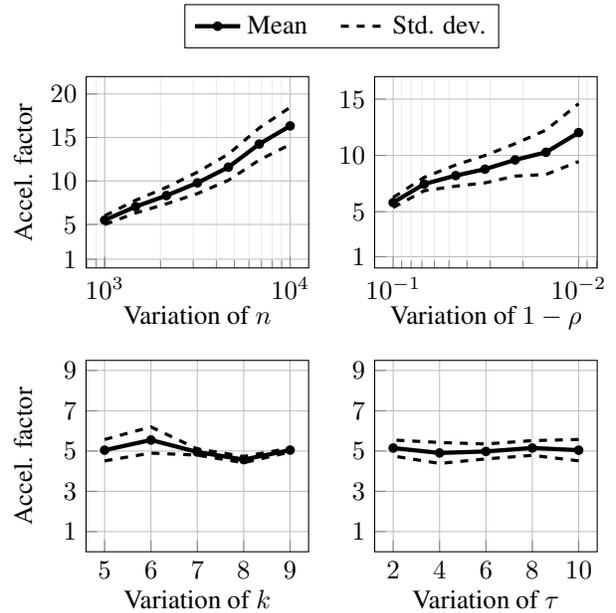
\begin{figure}[!t]
    \centering
    \pgfplotscreateplotcyclelist{cycle_list_sensibility}{
    black, solid, ultra thick, mark=*, mark options={scale=0.5}\\
    black, dashed, very thick\\
    black, dashed, very thick\\
}

\begin{tikzpicture}
    \begin{groupplot}[
        group style         = {
            group size      = 2 by 2,
            horizontal sep  = 25pt,
            vertical sep    = 35pt,
            ylabels at      = edge left,
        },
        x label style       = {at={(0.5,-0.15)}},
        y label style       = {at={(-0.2,0.5)}},
        height              = 0.5 \linewidth,
        width               = 0.55 \linewidth,
        ytick pos           = left,
        xtick pos           = bottom,
        grid                = both,
        minor grid style    = {gray!15},
        major grid style    = {black!25},
        ylabel              = Accel. factor,
        cycle list name     = cycle_list_sensibility,
    ]

        \nextgroupplot[
            xlabel          = Variation of $\pdim$,
            xmode           = log,
            ytick           = {1, 5, 10, 15, 20},
            ymin            = 0,
            ymax            = 22,
            legend to name  = sensibility_legend,
            legend columns  = 2,
            legend style    = {/tikz/column 2/.style={column sep=10pt}},
        ]
        \addplot table[
            x       = n, 
            y       = acceleration_mean,
            col sep = comma,
        ] {data/sensibility_n.csv};
        \addplot table[
            x       = n, 
            y expr  = {\thisrow{acceleration_mean} + \thisrow{acceleration_std}},
            col sep = comma,
        ] {data/sensibility_n.csv};
        \addplot table[
            x       = n, 
            y expr  = {\thisrow{acceleration_mean} - \thisrow{acceleration_std}},
            col sep = comma,
        ] {data/sensibility_n.csv};
        \addlegendentry{Mean}
        \addlegendentry{Std. dev.}

        \nextgroupplot[
            xlabel  = Variation of $1 - \corrparam$,
            xmode   = log,
            x dir   = reverse,
            ytick   = {1, 5, 10, 15},
            ymin    = 0,
            ymax    = 17,
        ]
        \addplot table[
            x expr  = {1 - \thisrow{rho}}, 
            y       = acceleration_mean,
            col sep = comma,
        ] {data/sensibility_rho.csv};
        \addplot table[
            x expr  = {1 - \thisrow{rho}}, 
            y expr  = {\thisrow{acceleration_mean} + \thisrow{acceleration_std}},
            col sep = comma,
        ] {data/sensibility_rho.csv};
        \addplot table[
            x expr  = {1 - \thisrow{rho}}, 
            y expr  = {\thisrow{acceleration_mean} - \thisrow{acceleration_std}},
            col sep = comma,
        ] {data/sensibility_rho.csv};

        \nextgroupplot[
            xlabel  = Variation of $\sparsitylevel$,
            ytick   = {1, 3, 5, 7, 9},
            xtick   = {5, 6, 7, 8, 9},
            ymin    = 0,
            ymax    = 9.5,
        ]
        \addplot table[
            x       = k, 
            y       = acceleration_mean,
            col sep = comma
        ] {data/sensibility_k.csv};
        \addplot table[
            x       = k, 
            y expr  = {\thisrow{acceleration_mean} + \thisrow{acceleration_std}},
            col sep = comma,
        ] {data/sensibility_k.csv};
        \addplot table[
            x       = k, 
            y expr  = {\thisrow{acceleration_mean} - \thisrow{acceleration_std}},
            col sep = comma,
        ] {data/sensibility_k.csv};

        \nextgroupplot[
            xlabel  = Variation of $\snr$,
            ytick   = {1, 3, 5, 7, 9},
            ymin    = 0,
            ymax    = 9.5,
        ]
        \addplot table[
            x       = snr, 
            y       = acceleration_mean,
            col sep = comma
        ] {data/sensibility_snr.csv};
        \addplot table[
            x       = snr, 
            y expr  = {\thisrow{acceleration_mean} + \thisrow{acceleration_std}},
            col sep = comma,
        ] {data/sensibility_snr.csv};
        \addplot table[
            x       = snr, 
            y expr  = {\thisrow{acceleration_mean} - \thisrow{acceleration_std}},
            col sep = comma,
        ] {data/sensibility_snr.csv};
    \end{groupplot}
    \path (group c1r1.north east) -- node[above,yshift=0.25cm]{\ref{sensibility_legend}} (group c2r1.north west);
\end{tikzpicture}
    \caption{Acceleration factor when implementing the simultaneous pruning tests in addition to the standard pruning strategy during the \gls{bnb} algorithm.}
    \label{fig:sensibility}
\end{figure}

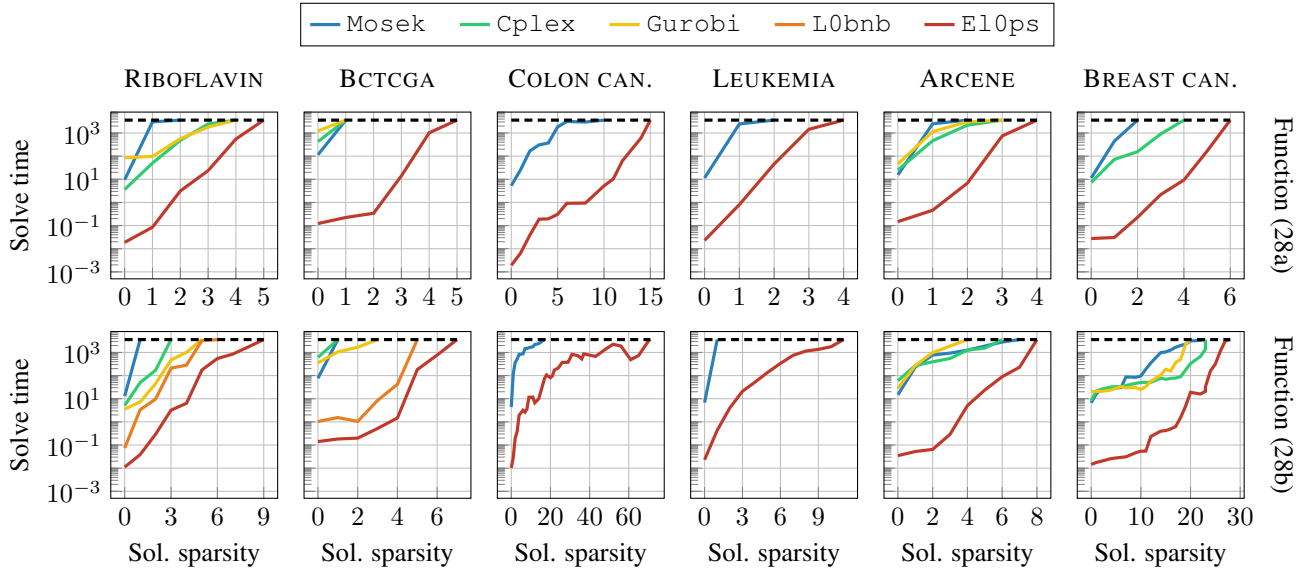
\begin{figure*}
    \centering
    \pgfplotscreateplotcyclelist{cycle_list_regpath_sparsity}{
    mosekcolor, very thick\\    
    cplexcolor, very thick\\
    gurobicolor, very thick\\
    l0bnbcolor, very thick\\
    bnbcolor, very thick\\
    black, densely dashed, very thick\\
}

\pgfplotsset{
 	legend image code/.code={
		\draw[mark repeat=2,mark phase=2] plot coordinates {
			(0cm,0cm)
			(0.4cm,0cm)
			(0.4cm,0cm)
		};
    }
}

\begin{tikzpicture}
    \begin{groupplot}[
        group style     = {
            group size      = 6 by 2,
            xlabels at      = edge bottom,
            yticklabels at  = edge left,
            vertical sep    = 20pt,
            horizontal sep  = 10pt,
        },
        height          = 3.8cm,
        width           = 3.8cm,
        ytick pos       = left,
        xtick pos       = bottom,
        grid            = major,
        ymode           = log,
        ymin            = 0.0005,
        ymax            = 8000,
        ytick           = {0.001,0.01,0.1,1,10,100,1000,10000},
        yticklabels     = {$10^{-3}$,,$10^{-1}$,,$10^{1}$,,$10^{3}$,},
        xlabel          = Sol. sparsity,
        cycle list name = cycle_list_regpath_sparsity,
    ]


        \nextgroupplot[
            title   = \textsc{Riboflavin},
            ylabel  = Solve time,
            xtick   = {0,1,2,3,4,5},
        ]
        \foreach \solver in {mosek,cplex,gurobi,l0bnb,proposed,max}{
            \addplot table[
                x       = n_nnz, 
                y       = \solver_solve_time,
                col sep = comma,
            ] {data/regpath_sparsity_riboflavin_l1bigm.csv};
        }

        \nextgroupplot[
            title   = \textsc{Bctcga},
            xtick   = {0,1,2,3,4,5},
        ]
        \foreach \solver in {mosek,cplex,gurobi,l0bnb,proposed,max}{
            \addplot table[
                x       = n_nnz, 
                y       = \solver_solve_time,
                col sep = comma,
            ] {data/regpath_sparsity_bctcga_l1bigm.csv};
        }

        \nextgroupplot[
            title   = \textsc{Colon can.},
        ]
        \foreach \solver in {mosek,cplex,gurobi,l0bnb,proposed,max}{
            \addplot table[
                x       = n_nnz, 
                y       = \solver_solve_time,
                col sep = comma,
            ] {data/regpath_sparsity_colon-cancer_l1bigm.csv};
        }

        \nextgroupplot[
            title   = \textsc{Leukemia},
            xtick   = {0,1,2,3,4},
        ]
        \foreach \solver in {mosek,cplex,gurobi,l0bnb,proposed,max}{
            \addplot table[
                x       = n_nnz, 
                y       = \solver_solve_time,
                col sep = comma,
            ] {data/regpath_sparsity_leukemia_l1bigm.csv};
        }

        \nextgroupplot[
            title   = \textsc{Arcene},
            xtick   = {0,1,2,3,4},
        ]
        \foreach \solver in {mosek,cplex,gurobi,l0bnb,proposed,max}{
            \addplot table[
                x       = n_nnz, 
                y       = \solver_solve_time,
                col sep = comma,
            ] {data/regpath_sparsity_arcene_l1bigm.csv};
        }

        \nextgroupplot[
            title           = \textsc{Breast can.},
            ylabel right    = Function \eqref{eq:bigml1},
        ]
        \foreach \solver in {mosek,cplex,gurobi,l0bnb,proposed,max}{
            \addplot table[
                x       = n_nnz, 
                y       = \solver_solve_time,
                col sep = comma,
            ] {data/regpath_sparsity_breast-cancer_l1bigm.csv};
        }


        \nextgroupplot[
            ylabel          = Solve time,
            xtick           = {0,3,6,9},
            legend to name  = regpath_sparsity_legend,
            legend style    = {
                legend columns = 6,
                /tikz/every even column/.style = {column sep=10pt},
            },  
        ]
        \foreach \solver in {mosek,cplex,gurobi,l0bnb,proposed,max}{
            \addplot table[
                x       = n_nnz, 
                y       = \solver_solve_time,
                col sep = comma,
            ] {data/regpath_sparsity_riboflavin_l2bigm.csv};
        }
        \addlegendentry{\texttt{Mosek}}
        \addlegendentry{\texttt{Cplex}}
        \addlegendentry{\texttt{Gurobi}}
        \addlegendentry{\texttt{L0bnb}}
        \addlegendentry{\texttt{El0ps}}

        \nextgroupplot
        \foreach \solver in {mosek,cplex,gurobi,l0bnb,proposed,max}{
            \addplot table[
                x       = n_nnz, 
                y       = \solver_solve_time,
                col sep = comma,
            ] {data/regpath_sparsity_bctcga_l2bigm.csv};
        }

        \nextgroupplot
        \foreach \solver in {mosek,cplex,gurobi,l0bnb,proposed,max}{
            \addplot table[
                x       = n_nnz, 
                y       = \solver_solve_time,
                col sep = comma,
            ] {data/regpath_sparsity_colon-cancer_l2bigm.csv};
        }

        \nextgroupplot[
            xtick   = {0,3,6,9},
        ]
        \foreach \solver in {mosek,cplex,gurobi,l0bnb,proposed,max}{
            \addplot table[
                x       = n_nnz, 
                y       = \solver_solve_time,
                col sep = comma,
            ] {data/regpath_sparsity_leukemia_l2bigm.csv};
        }

        \nextgroupplot
        \foreach \solver in {mosek,cplex,gurobi,l0bnb,proposed,max}{
            \addplot table[
                x       = n_nnz, 
                y       = \solver_solve_time,
                col sep = comma,
            ] {data/regpath_sparsity_arcene_l2bigm.csv};
        }

        \nextgroupplot[
            ylabel right    = Function \eqref{eq:bigml2},
        ]
        \foreach \solver in {mosek,cplex,gurobi,l0bnb,proposed,max}{
            \addplot table[
                x       = n_nnz, 
                y       = \solver_solve_time,
                col sep = comma,
            ] {data/regpath_sparsity_breast-cancer_l2bigm.csv};
        }
    \end{groupplot}
    \path (group c1r1.north east) -- node[above,yshift=0.75cm]{\ref{regpath_sparsity_legend}} (group c6r1.north west);
\end{tikzpicture}
    \caption{Time to construct solutions with a given sparsity level. The black dotted line represents the maximum time of one hour allowed.}
    \label{fig:regpath-sparsity}
\end{figure*}

We remark that the gains obtained by our methodology increase with the dimension parameter $\pdim$ and the correlation parameters $\corrparam$.
In contrast, the gain does not seem to be significantly impacted by the sparsity level $\sparsitylevel$ or the signal-to-noise ratio $\snr$. 
This suggests that depending on the characteristics of the problem, the proposed pruning methodology can lead to different gains in terms of running time.
We note nonetheless that in all the tested scenarios, the solving time is improved by at least a factor $5$.

\subsection{Performance on Real-World Datasets}
\label{sec:numerics:realworld}

In this section, we assess the proposed pruning methodology on six real-world datasets.

\paragraph{Problems and datasets} 
We address feature selection problems, which correspond to typical machine learning tasks. 
Each task corresponds to an instance of the loss function $\lossfunc$ and two choices of $\pertfunc$, as described below. 
Each dataset provides a matrix \(\dic \in \kR^{\ddim \times \pdim}\) and a vector \(\obs \in \kR^{\ddim}\).
The dimension of each dataset is specified in \Cref{tab:datasets}.

First, we consider linear regression with the least-squares loss function:
\begin{equation} \label{eq:leastsquares-loss}
    \lossfunc(\cdot) = \tfrac{1}{2}\norm{\obs - \cdot}{2}^2
    .
\end{equation}
We use the instances of \(\obs\) and \(\dic\) provided by the \textsc{Riboflavin}~\cite{buhlmann2014high} and \textsc{Bctcga}~\cite{liu2018integrated} datasets which are related to vitamin production and cancer screening, respectively.

Second, we consider binary classification tasks with the logistic loss function:
\begin{equation} \label{eq:logistic-loss}
    \lossfunc(\cdot) = \transp{\1}\log(\1 + \exp(-\obs \odot \cdot))
\end{equation}
where $\odot$ denotes the Hadamard product and the functions $\log(\cdot)$ and $\exp(\cdot)$ are taken component-wise.
We use instances of \(\obs\) and \(\dic\) from the \textsc{Colon cancer}~\cite{alon1999broad} and~\textsc{Leukemia} \cite{golub1999molecular} datasets related to cancer screening.

Finally, we consider binary classification tasks with the squared hinge loss:
\begin{equation} \label{eq:squaredhinge-loss}
    \lossfunc(\cdot) = \norm{\pospart{\1 - \obs \odot \cdot}}{2}^2
\end{equation}
where $\pospart{\cdot}$ is taken component-wise.
We use instances of \(\obs\) and \(\dic\) from the datasets \textsc{Breast cancer} and \textsc{Arcene} \cite{chang2011libsvm} related to DNA analysis and tumor categorization, respectively.

\begin{table}[!t]
    \center
    \begin{tabular}{ccrr}
    \toprule
    Dataset & \(\lossfunc\) & $\ddim$ & $\pdim$ \\ 
    \midrule
    \textsc{Riboflavin} & Least-squares & 71 & 4,088 \\
    \textsc{Bctcga} & Least-squares & 536 & 17,322 \\
    \textsc{Colon cancer} & Logistic & 62 & 2,000 \\
    \textsc{Leukemia} & Logistic & 38 & 7,129 \\
    \textsc{Breast cancer} & Squared-hinge & 44 & 7,129 \\
    \textsc{Arcene} & Squared-hinge & 100 & 10,000 \\
    \bottomrule
\end{tabular}
    \caption{Dimensions of the datasets and data fidelity term \(\lossfunc\).}
    \label{tab:datasets}
\end{table}

For each dataset, we consider the two following choices:
\begin{subequations}
    \begin{alignat}{4}
        \label{eq:bigml1}
        \pertfunc(\cdot) &= \alpha\abs{\cdot} &&+ \icvx(\abs{\cdot} \leq \bigM) \\
        \label{eq:bigml2}
        \pertfunc(\cdot) &= \alpha\abs{\cdot}^2 &&+ \icvx(\abs{\cdot} \leq \bigM)
    \end{alignat}
\end{subequations}
where $\alpha>0$ and $\bigM>0$.
These choices are motivated by the statistical properties of the solutions that can be obtained \cite{dedieu2021learning}.

We mention that the \texttt{Cplex} and \texttt{Gurobi} solvers can only handle linear and quadratic functions. Hence, they cannot address problem instances involving the logistic loss \eqref{eq:logistic-loss}.
Moreover, \texttt{L0bnb} can only handle instances combining functions \eqref{eq:leastsquares-loss} and \eqref{eq:bigml2}.

\paragraph{Performance profiles}

For each problem instance, we first calibrate the hyperparameters $\alpha>0$ and $\bigM>0$ as explained in \Cref{sec:supp_numerics:hyperparameters}.
We then fit a regularization path \cite{friedman2010regularization}, that is, we vary the value of $\reg$ to construct solutions with different sparsity levels.
We start at some $\reg$ so that the all-zero vector is a solution of \eqref{prob:prob} and sequentially decrease its value as long as at least one solver can solve the problem within one hour.
The solution obtained for each value of $\reg$ is used as a warm-start for the problem with the next value of $\reg$ considered in the regularization path.

\Cref{fig:regpath-sparsity} represents the time needed by each solver to construct a solution with a given sparsity level.
We observe that the implementation of the pruning methodology proposed in this paper allows for significant gains in terms of running time.
More precisely, our method outperforms the other solvers in all the considered scenarios. 
In comparison with off-the-shelf \gls{mip} solvers such as \texttt{Mosek}, \texttt{Cplex}, and \texttt{Gurobi}, the time savings can reach up to four orders of magnitude in the most favorable cases. 
Regarding the specialized solver \texttt{L0bnb}, improvements of up to two orders of magnitude are achievable by our method in the best-case scenarios.


\section{Conclusion}

In this paper, we introduce a new methodology to perform pruning tests in Branch-and-Bound algorithms addressing $\ell_0$-regularized optimization problems. 
Our method is only grounded on a few hypotheses and can thus be applied to a large variety of problems, notably in machine learning. 
Our numerical results demonstrate that the proposed methodology significantly reduces solving time compared to other state-of-the-art methods. 
It therefore allows to address some problem instances that were out of computational reach so far.

\section*{Impact Statement}
The goal of our work is to accelerate the solving time of some particular optimization problems.   
Our contribution is primarily methodological and any of its potential societal impact would only be indirectly related to our work.

\bibliography{main}
\bibliographystyle{icml2024}

\newpage
\appendix
\onecolumn

\section{Supplementary Material Related to \texorpdfstring{\Cref{sec:bnb,sec:screening}}{Sections~\ref{sec:bnb} and~\ref{sec:screening}}} \label{sec:proofs}

This section gathers discussions and proofs of the results presented in \Cref{sec:bnb,sec:screening} of the paper.

\subsection{Discussion on \texorpdfstring{\Cref{sec:bnb:complexity}}{Section~\ref{sec:bnb:complexity}}}
\label{sec:proofs:complexity}

{
\newcommand{\idxdepth}{\ell}
\newcommand{\criticalidx}{\idxdepth_c}
\newcommand{\initialnode}{\nodeSymb_{0}}
\newcommand{\sequencenode}[1]{\nodeSymb_{#1}}
\newcommand{\lengthnone}{L}

In this paragraph, we give additional details on the relation between the tightness of lower-bound \eqref{eq:std-lb} and the depth of the node at which it is computed.
Our claim is grounded on the two results stated in the following lemma:
\begin{lemma} \label{lemma appendix discussion}
    Let \(\nodeSymb=(\setone,\setzero,\setnone)\) be a node of the \gls{bnb} tree. 
    Then
    \begin{enumerate}
        \item For all successors \(\nodeSymb'\) of \(\nodeSymb\), we have \(\robj^{\nodeSymb'} \geq \robj^{\nodeSymb}\).
        \item If \ref{hyp:pertfunc}-\ref{hyp:zero-minimized} hold and \(\setnone=\emptyset\) then \(\node{\robj} = \node{\pobj}\).
        \label{item:lemma appendix discussion:robj = pobj}
    \end{enumerate}
\end{lemma}
\begin{proof}
    We prove the two items separately.
    \begin{enumerate}
        \item Let \(\nodeSymb'=(\setone',\setzero',\setnone')\) be a successor of \(\nodeSymb\).
        By definition of a successor, we observe from~\eqref{eq:region} that \(\pset^{\nodeSymb'} \subseteq \node{\pset}\).
        Therefore, using the definitions of \({\regfunc}^{\nodeSymb}\) and \({\regfunc}^{\nodeSymb'}\) in~\eqref{eq:node-regfunc}, we have 
        \begin{equation} \label{eq:proof bounds get tighter:central inequality}
            {\regfunc}^{\nodeSymb'}(\cdot)
            \geq
            {\regfunc}^{\nodeSymb}(\cdot) 
            .
        \end{equation}
        \label{item:lemma appendix discussion:robj is increasing}
        Item~\textit{(ii)} of Proposition~13.16 in \cite{bauschke2017convex} then leads to
        \begin{equation}
            \biconj{(\regfunc^{\nodeSymb'})}(\cdot)
            \geq
            \biconj{(\node{\regfunc})}(\cdot)
            .
        \end{equation}
        Thus, we have the inequality \(\lossfunc(\dic\cdot) + \biconj{(\regfunc^{\nodeSymb'})}(\cdot) \geq \lossfunc(\dic\cdot) + \biconj{(\regfunc^{\nodeSymb})}(\cdot)\).
        By taking the infimum on both sides of this inequality, we obtain the desired result.

        \item If \ref{hyp:pertfunc}-\ref{hyp:zero-minimized} hold and \(\setnone=\emptyset\), one has \(\biconj{(\node{\regfunc})} = \node{\regfunc}\) from \Cref{prop:dual-relaxregfunc} and Theorem 4.8 in \citep{beck2017first}. Hence, the result directly follows. 
    \end{enumerate}
\end{proof}
We now motivate our claim ``achieving the prescribed tightness usually imposes exploring \emph{deep} nodes in the decision tree'' in light of the following analysis.

Let \(\UB{\pobj}\) be an upper bound on \(\opt{\pobj}\) and \(\initialnode=(\setone^{(0)},\setzero^{(0)},\setnone^{(0)})\) be a node\footnote{
    Contrary to the notational convention used in~\Cref{sec:bnb:exploration}, \(\initialnode\) does not necessarily refers (here) to the root node \((\emptyset,\emptyset,\intervint{1}{\pdim})\).
} of the \gls{bnb} tree satisfying \(\setnone^{(0)}\neq\emptyset\) and
\begin{equation} \label{eq:appendix discussion:node passes ideal test}
    \pobj^{\initialnode} 
    > 
    \UB{\pobj}
    \geq 
    \robj^{\initialnode}
    .
\end{equation}
Denote \(\lengthnone=\card{\setnone^{(0)}}\).
The first inequality guarantees that \(\pset^{\nodeSymb_0}\) definitively excludes any minimizer of~\eqref{prob:prob}, while the second indicates that the pruning test~\eqref{eq:pruning-test} with the lower bound~\eqref{eq:std-lb} is unsuccessful.
We nevertheless demonstrate that  for any sequence \(\{\sequencenode{\idxdepth}\}_{\idxdepth=1}^{\lengthnone}\)
satisfying 
\begin{equation}\label{eq:cond-node-sequence}
    \forall\idxdepth\in\intervint{0}{\lengthnone-1}:\quad
    \sequencenode{\idxdepth+1} \text{ is a direct successor of } \sequencenode{\idxdepth}
    ,
\end{equation}
there exists an index \(\criticalidx\in\intervint{1}{\lengthnone}\) 
such that \(\sequencenode{\criticalidx}\) passes test~\eqref{eq:pruning-test} with the lower bound~\eqref{eq:std-lb}.

Let \(\{\sequencenode{\idxdepth}\}_{\idxdepth=1}^{\lengthnone}\) be a sequence verifying \eqref{eq:cond-node-sequence}. 
First, using item~\ref{item:lemma appendix discussion:robj is increasing} of \Cref{lemma appendix discussion}, we have that $\delta_\idxdepth \triangleq \robj^{\sequencenode{\idxdepth}} - \UB{\pobj}$ is a non-decreasing function of $\idxdepth$. 
Note that the second inequality in \eqref{eq:appendix discussion:node passes ideal test} leads to \(\delta_0 \leq 0\).
Moreover, we have that \(\delta_{\lengthnone} > 0\) by virtue of the following arguments: first, the definition of a successor implies that \(\pset^{\sequencenode{\lengthnone}} \subsetneq \pset^{\initialnode}\) and therefore \(\pobj^{\sequencenode{\lengthnone}} \geq \pobj^{\initialnode}\); second, we have from item~\ref{item:lemma appendix discussion:robj = pobj} of \Cref{lemma appendix discussion} that \(\pobj^{\nodeSymb_{\lengthnone}} = \robj^{\sequencenode{\lengthnone}}\); finally the result follows from the first inequality in \eqref{eq:appendix discussion:node passes ideal test}.

In conclusion, since \(\delta_{\lengthnone} > 0\), the pruning test~\eqref{eq:pruning-test} with the lower bound~\eqref{eq:std-lb} will ultimately succeed if enough successors are visited by the \gls{bnb} procedure.
Furthermore, the likelihood of the success increases as the \gls{bnb} procedure explores deeper nodes in the tree because \(\delta_\idxdepth\) is a non-decreasing function of \(\idxdepth\).
}

\subsection{Proof of \texorpdfstring{\Cref{prop:dual-relaxregfunc}}{Proposition~\ref{prop:dual-relaxregfunc}}} \label{sec:proof:dual-regfunc}

We first prove the following technical lemma:
{
\newcommand{\idxsequence}{i}
\newcommand{\proofsequence}[1]{\pvi{}^{(#1)}}
\newcommand{\proofsubsequence}[1]{\widetilde{\pvi{}}^{(#1)}}
\begin{lemma}
    \label{lemma:infima=}
    Let $\kfuncdef{\prooffun}{\kR}{\kR\cup\{+\infty\}}$ be a closed, convex, proper function and  
    let \(\pvi{0}\in\dom(\prooffun)\) be an accumulation point of \(\dom(\prooffun)\).
    Then we have 
    \begin{equation}
        \inf_{\pvi{}\in\kR\setminus\{\pvi{0}\}} \prooffun(\pvi{})
        =
        \inf_{\pvi{}\in\kR} \prooffun(\pvi{}).
    \end{equation}
\end{lemma}
\begin{proof}
    We obviously have
    \begin{equation}
        \inf_{\pvi{}\in\kR\setminus\{\pvi{0}\}} \prooffun(\pvi{})
        \geq
        \inf_{\pvi{}\in\kR} \prooffun(\pvi{})
        ,
    \end{equation}
    so we concentrate on the reverse inequality hereafter. 
    
    First, using the fact that \(\pvi{0}\) is an accumulation point of \(\dom(\prooffun)\), we have that \(\dom(\prooffun)\setminus\{\pvi{0}\}\neq \emptyset\) and therefore
    \begin{align}
        \inf_{\pvi{}\in\kR} \prooffun(\pvi{}) &=  \inf_{\pvi{}\in\dom(\prooffun)} \prooffun(\pvi{})
        \\
        \inf_{\pvi{}\in\kR\setminus\{\pvi{0}\}} \prooffun(\pvi{}) &=  \inf_{\pvi{}\in\dom(\prooffun)\setminus\{\pvi{0}\}} \prooffun(\pvi{})
        .
    \end{align}
    It is thus sufficient to prove that
    \begin{equation} \label{eq:proof lemma infima=:target inequality}
        \inf_{\pvi{}\in\dom(\prooffun)\setminus\{\pvi{0}\}} \prooffun(\pvi{})
        \leq
        \inf_{\pvi{}\in\dom(\prooffun)} \prooffun(\pvi{})
        .
    \end{equation}
    Second, using the fact that \(\pvi{0}\in\dom(\prooffun)\) by hypothesis, we also have
    \begin{equation} \label{eq:proof lemma infima=:infima in not +infty}
        \inf_{\pvi{} \in\dom(\prooffun)} \prooffun(\pvi{})
        \leq \prooffun(\pvi{0})
        < +\infty
        .
    \end{equation}
    We then prove~\eqref{eq:proof lemma infima=:target inequality} by considering two separate cases. 
    \begin{itemize}
        \item Assume first that 
        \begin{equation} \label{eq:proof lemma infima=:infima in not +infty:case 1}
            \inf_{\pvi{} \in\dom(\prooffun)} \prooffun(\pvi{})
            =
            \prooffun(\pvi{0})
            .
        \end{equation}
        Since \(\pvi{0}\) is an accumulation point of \(\dom(\prooffun)\), there exists a sequence \(\{\pvi{}^{(\idxsequence)}\}_{\idxsequence\in\kN}\subset\dom(\prooffun)\setminus\{\pvi{0}\}\).\footnote{
            \label{footnote:construction de suite convergeant vers x0}
            Such a sequence can be constructed as follows: for all \(\idxsequence\in\kN\), one chooses \(\pvi{}^{(\idxsequence)}\in(\proofset\cap B(\pvi{0}, \kinv{(\idxsequence+1)}))\setminus\{\pvi{0}\}\) which is nonempty by definition of an accumulation point. 
            }
        From Th. 2.22 in \cite{beck2017first}, we have that closedness, convexity and properness of \(\prooffun\) implies that it is continuous on its domain.\footnote{
            That is for any $\{x^{(\idxsequence)}\}_{\idxsequence\in\kN} \subset \dom(\prooffun)$ converging to some limit point $\pvi{0}\in\dom(\prooffun)$, we have $\lim_{\idxsequence\rightarrow+\infty} \prooffun(\pvi{}^{(\idxsequence)})=\prooffun(\pvi{0})$.
        }
        Hence \(\lim_{\idxsequence\rightarrow+\infty} \prooffun(\pvi{}^{(\idxsequence)})=\prooffun(\pvi{0})\) and therefore
        \begin{equation} \label{eq:proof lemma infima=:infima in not +infty:case 1 buf}
            \inf_{\pvi{}\in\dom(\prooffun)\setminus\{\pvi{0}\}}  \prooffun(\pvi{})
            \leq
            \prooffun(\pvi{0})
            .
        \end{equation}
        Inequality~\eqref{eq:proof lemma infima=:target inequality} immediately follows by combining~\eqref{eq:proof lemma infima=:infima in not +infty:case 1} and~\eqref{eq:proof lemma infima=:infima in not +infty:case 1 buf}.

        \item Assume now that
        \begin{equation} \label{eq:proof lemma infima=:infima in not +infty:case 2}
            \inf_{\pvi{} \in\dom(\prooffun)} \prooffun(\pvi{})
            <
            \prooffun(\pvi{0})
        \end{equation}
        and denote \(\prooffun_{\inf}\in\kR\cup\{-\infty\}\) the latter infimum.
        By definition of an infimum, there exists a sequence \(\{\proofsequence{\idxsequence}\}_{\idxsequence\in\kN}\subset\dom(\prooffun)\) such that \(\prooffun_{\inf} = \lim_{\idxsequence\rightarrow+\infty} \prooffun(\proofsequence{\idxsequence})\).
        Let \(\varepsilon>0\) be such that \(\varepsilon<\prooffun(\pvi{0}) - \prooffun_{\inf}\).
        The hypothesis case~\eqref{eq:proof lemma infima=:infima in not +infty:case 2} therefore implies that there exists \(\idxsequence_{\varepsilon}\) such that
        \begin{equation}
            \forall\idxsequence\in\kN,\qquad
            \idxsequence\geq\idxsequence_\varepsilon \Longrightarrow
            \prooffun(\proofsequence{\idxsequence}) < \prooffun(\pvi{0}) - \varepsilon
        \end{equation}
        and therefore \(\proofsequence{\idxsequence}\neq\pvi{0}\) for all \(\idxsequence\geq\idxsequence_\varepsilon\).
        On can thus construct a subsequence \(\{\proofsubsequence{\idxsequence}\}_{\idxsequence\in\kN}\) such that 
        \(\proofsubsequence{\idxsequence}\neq\pvi{0}\) for all \(\idxsequence\in\kN\).
        This implies that \(\{\proofsubsequence{\idxsequence}\}_{\idxsequence\in\kN}\subset\dom(\prooffun)\setminus\{\pvi{0}\}\) and one immediately deduces that
        \begin{equation}
            \inf_{\pvi{}\in\dom(\prooffun)\setminus\{\pvi{0}\}}
            \leq
            \lim_{\idxsequence\to+\infty} \prooffun(\proofsubsequence{\idxsequence}) =
            \prooffun_{\inf}
        \end{equation}
        where the equality holds since \(\{\prooffun(\proofsubsequence{\idxsequence})\}_{\idxsequence\in\kN}\) is a subsequence of a converging sequence.
        This leads to~\eqref{eq:proof lemma infima=:target inequality}. 
    \end{itemize}
\end{proof}
}

We are now ready to give a proof to \Cref{prop:dual-relaxregfunc}. 
By definition of a convex conjugate function, we have: 
\begin{equation}
    \forall\vv \in \kR^{\pdim}:\ \conj{(\node{\regfunc})}(\vv) = \textstyle\sup_{\pv \in \kR^{\pdim}} \transp{\vv}\pv - \node{\regfunc}(\pv)
    .
\end{equation}
Observing from \eqref{eq:node-regfunc} that $\node{\regfunc}(\cdot)$ is separable since both $\regfunc(\cdot)$ and $\icvx(\cdot \in \node{\pset})$ are separable (see their definitions in~\eqref{eq:regfunc} and~\eqref{eq:region}, respectively), we deduce that $\conj{(\node{\regfunc})}(\cdot)$ is also separable and is given coordinate-wise by
\begin{align}
    \conj{(\separable{\node{\regfunc}}{\idxentry})}(\vvi{}) &= \textstyle\sup_{\pvi{} \in \kR} \vvi{}\pvi{} - \separable{\node{\regfunc}}{\idxentry}(\pvi{})
    .
\end{align}
Imposing explicitly the constraints defined in $\node{\pset}$, we obtain: 
\begin{align}
    \conj{(\separable{\node{\regfunc}}{\idxentry})}(\vvi{}) &= 
    \begin{cases}
        \textstyle\sup_{\pvi{} = 0} \, \vvi{}\pvi{} - \pertfunc(\pvi{}) - \reg\norm{\pvi{}}{0} &\text{if} \ \idxentry \in \setzero \\
        \textstyle\sup_{\pvi{} \neq 0} \, \vvi{}\pvi{} - \pertfunc(\pvi{}) - \reg\norm{\pvi{}}{0} &\text{if} \ \idxentry \in \setone \\
        \textstyle\sup_{\pvi{} \in \kR} \vvi{}\pvi{} - \pertfunc(\pvi{}) - \reg\norm{\pvi{}}{0} &\text{if} \ \idxentry \in \setnone
    \end{cases} 
    \label{eq:proof:dual-relaxregfunc:2}
\end{align}
We next address the three above cases separately.

If \(\idxentry \in \setzero\), the first case in~\eqref{eq:proof:dual-relaxregfunc:2} simplifies to
\begin{equation} \label{eq:proof:dual-relaxregfunc:3}
    \conj{(\separable{\node{\regfunc}}{\idxentry})}(\vvi{}) = \textstyle\sup_{\pvi{} = 0} \, \vvi{}\pvi{} - \pertfunc(\pvi{}) = 0
\end{equation}
where the first equality holds since \(\norm{0}{0}=0\) and the second since $\pertfunc(0)=0$ in virtue of hypothesis~\ref{hyp:zero-minimized}.

If $\idxentry \in \setone$, then
\begin{align}
    \label{eq:proof:dual-relaxregfunc:4}
    \conj{(\separable{\node{\regfunc}}{\idxentry})}(\vvi{}) &= \textstyle\sup_{\pvi{} \neq 0} \ \, \vvi{}\pvi{} - \pertfunc(\pvi{}) - \reg 
\end{align}
since \(\norm{\pvi{}}{0}=1\) for all \(\pvi{}\neq0\). 
As $\pertfunc(\cdot)$ is closed, convex and proper, the function $\prooffun(\pvi{})\triangleq -\vvi{}\pvi{} + \pertfunc(\pvi{}) + \reg$ inherits from these properties. On the one hand, it is easy to see that $\dom (\prooffun)=\dom(\pertfunc)$.
On the other hand, since $0$ is an accumulation point of $\dom(\pertfunc)$ from \ref{hyp:0-in-intdom}, we have from \Cref{lemma:infima=} that
\begin{align}
    \conj{(\separable{\node{\regfunc}}{\idxentry})}(\vvi{}) 
    \label{eq:proof:dual-relaxregfunc:5}
    &= \textstyle\sup_{\pvi{} \in \kR} \, \vvi{}\pvi{} - \pertfunc(\pvi{}) - \reg. 
\end{align}
We finally obtain the result by using the definition of the convex conjugate of $\pertfunc$. 

If $\idxentry \in \setnone$, then
\begin{subequations}
    \begin{alignat}{4}
        \conj{(\separable{\node{\regfunc}}{\idxentry})}(\vvi{}) &= \textstyle\sup_{\pvi{} \in \kR} \, \vvi{}\pvi{} - \pertfunc(\pvi{})  - \reg\norm{\pvi{}}{0} \\
        &= \max\big\{
            \textstyle\sup_{\pvi{} = 0} \, \vvi{}\pvi{} - \pertfunc(\pvi{}) - \reg\norm{\pvi{}}{0} \ &&, \
            \textstyle\sup_{\pvi{} \neq 0} \, \vvi{}\pvi{} - \pertfunc(\pvi{}) - \reg\norm{\pvi{}}{0}
        &&\big\} \\
        \label{eq:proof:dual-relaxregfunc:7}
        &= \max\big\{
            \textstyle\sup_{\pvi{} = 0} \, \vvi{}\pvi{} - \pertfunc(\pvi{}) \ &&, \
            \textstyle\sup_{\pvi{} \neq 0} \, \vvi{}\pvi{} - \pertfunc(\pvi{}) - \reg
        &&\big\} \\
        \label{eq:proof:dual-relaxregfunc:8}
        &= \max \left\{0, \conj{\pertfunc}(\vvi{}) - \reg \right\} \\
        \label{eq:proof:dual-relaxregfunc:9}
        &= \pospart{\conj{\pertfunc}(\vvi{}) - \reg}
    \end{alignat}
\end{subequations}
where \eqref{eq:proof:dual-relaxregfunc:7} is obtained by definition of the $\ell_0$-norm, \eqref{eq:proof:dual-relaxregfunc:8} follows from the same reasoning as for the case ``$\idxentry\in\setone$''. 


\subsection{Proof of \texorpdfstring{\Cref{prop:dual-link}}{Proposition~\ref{prop:dual-link}}} \label{sec:proof:dual-link}

{
    \newcommand{\lemmaNodeSymb}{\nodeSymb}
    \newcommand{\lemmaSetzero}{\setzero}
    \newcommand{\lemmaSetone}{\setone}
    \newcommand{\lemmaSetnone}{\setnone}

    \newcommand{\proofChildNodeSymb}{\nodeSymb'}
    \newcommand{\proofChildSetzero}{\setzero'}
    \newcommand{\proofChildSetone}{\setone'}
    \newcommand{\proofChildSetnone}{\setnone'}

    \newcommand{\subnodedepth}{k}
    \newcommand{\veryspecialentry}{\idxentry_0}

Our proof of \Cref{prop:dual-link} leverages the following relation between the dual functions at a node and its direct successors: 
\begin{lemma} \label{lemma:dual-link direct successor}
    Let \(\nodeSymb = (\setzero,\setone,\setnone)\) and \(\nodeSymb' = (\setzero',\setone',\setnone')\) be two nodes of the \gls{bnb} tree.

    If \(\nodeSymb'\) is a direct successor of \(\nodeSymb\), then for all \(\dv\in\kR^\ddim\): 
    \begin{equation} \label{eq:lemma:dual-link direct successor}
        \dfunc^{\nodeSymb'}(\dv)
        =
        \node{\dfunc}(\dv)
        +
        \begin{cases}
            \pivot{0}(\transp{\atom{\idxentry}}\dv) &\text{if} \ \idxentry \in \setzero' \\
            \pivot{1}(\transp{\atom{\idxentry}}\dv) &\text{if} \ \idxentry \in \setone'
        \end{cases}
    \end{equation}
    where \(\idxentry\) denotes the unique element of \((\setzero' \setminus \setzero) \cup (\setone' \setminus \setone)\) defined in \eqref{eq:direct-successor}.
\end{lemma}
The proof of this result is postponed to the end of the section.

\begin{proof}[Proof of \Cref{prop:dual-link}]
Let \(\nodeSymb = (\setzero,\setone,\setnone)\) and \(\dv\in\kR^\ddim\).
We show that~\eqref{eq:dual-function} is true by induction on the cardinality of \(\setnone\setminus\setnone'\) where \(\setnone'\) denotes the third element in the partition of a successor \(\nodeSymb'\) of \(\nodeSymb\).
More specifically, we show that
\begin{center}
    \(\forall\subnodedepth\in\intervint{0}{\card{\setnone}}\): ``For all successor node \(\proofChildNodeSymb=(\proofChildSetzero,\proofChildSetone,\proofChildSetnone)\) such that \(\card{\setnone\setminus\proofChildSetnone}=\subnodedepth\),~\eqref{eq:dual-function} holds true''.
 \end{center} 

\textit{Initialization.}
If \(\subnodedepth=0\), the only successor \(\proofChildNodeSymb\) of \(\nodeSymb\) satisfying \(\card{\setnone\setminus\proofChildSetnone}=\subnodedepth\) is \(\proofChildNodeSymb=\nodeSymb\).
In that case, \(\proofChildSetone\setminus\setone=\emptyset\) and \(\proofChildSetzero\setminus\setzero=\emptyset\) so that~\eqref{eq:dual-function} trivially holds. 

\textit{Induction.}
Let \(\subnodedepth\in\intervint{0}{\card{\setnone}-1}\) and assume that our induction hypothesis holds for \(\subnodedepth\).
Let also \(\proofChildNodeSymb=(\proofChildSetzero,\proofChildSetone,\proofChildSetnone)\) be a successor of \(\nodeSymb\) such that \(\card{\setnone\setminus\proofChildSetnone}=\subnodedepth+1\).
Since \(\setnone\setminus\proofChildSetnone\neq\emptyset\), we can choose \(\veryspecialentry\in\setnone\setminus\proofChildSetnone\) and define
\begin{equation}
    \nodeSymb^{\veryspecialentry}
    =
    (\proofChildSetzero\setminus\{\veryspecialentry\},\proofChildSetone\setminus\{\veryspecialentry\},\proofChildSetnone\cup\{\veryspecialentry\})
    .
\end{equation}
On the one hand, the definition of \(\nodeSymb^{\veryspecialentry}\) implies that \(\nodeSymb^{\veryspecialentry}\) is a successor of \(\nodeSymb\).
Hence, our induction hypothesis applied to \(\nodeSymb^{\veryspecialentry}\) leads to
\begin{equation} \label{eq:proof dual link:induction:intermediary expansion}
    \dfunc^{\nodeSymb^{\veryspecialentry}}(\dv) 
    = 
    \node{\dfunc}(\dv) 
    + 
    \sum_{\idxentry \in \setzero'\setminus(\setzero\cup\{\veryspecialentry\})} \pivot{0}(\transp{\atom{\idxentry}}\dv)
    +  \sum_{\idxentry \in \setone'\setminus(\setone\cup\{\veryspecialentry\})} \pivot{1}(\transp{\atom{\idxentry}}\dv)
    .
\end{equation}
On other hand, the definition of \(\nodeSymb^{\veryspecialentry}\) also implies that \(\proofChildNodeSymb\) is a direct successor of \(\nodeSymb^{\veryspecialentry}\).
Applying \Cref{lemma:dual-link direct successor} thus leads to
\begin{equation} \label{eq:proof dual link:induction:direct parent expansion}
    \dfunc^{\proofChildNodeSymb}(\dv)
    =
    \dfunc^{\nodeSymb^{\veryspecialentry}}(\dv) 
    +
    \begin{cases}
        \pivot{0}(\transp{\atom{\veryspecialentry}}\dv) &\text{if} \ \veryspecialentry \in \proofChildSetzero \\
        \pivot{1}(\transp{\atom{\veryspecialentry}}\dv) \qquad \ \ \ \, &\text{if} \ \veryspecialentry \in \proofChildSetone.
    \end{cases}
\end{equation}
One finally obtains~\eqref{eq:dual-function}  by expanding \(\dfunc^{\nodeSymb^{\veryspecialentry}}(\dv) \) in~\eqref{eq:proof dual link:induction:direct parent expansion} using the result of~\eqref{eq:proof dual link:induction:intermediary expansion} and noting that \(\{\veryspecialentry\}\cup\setzero'\setminus(\setzero\cup\{\veryspecialentry\}) = \setzero'\setminus\setzero\) and \(\{\veryspecialentry\}\cup\setone'\setminus(\setone\cup\{\veryspecialentry\}) = \setone'\setminus\setone\). 
Since this rationale holds irrespective of the successor $\proofChildNodeSymb$, we conclude that the induction hypothesis also holds for $\subnodedepth+1$, thereby completing the proof.
\end{proof}

\begin{proof}[Proof of \Cref{lemma:dual-link direct successor}] 
    We first expand the definition of the function associated to the dual problem at a given node (see~\eqref{prob:dual-node}).
    More specifically, we have for any node \(\lemmaNodeSymb=(\lemmaSetzero,\lemmaSetone,\lemmaSetnone)\) and \(\dv \in \kR^{\ddim}\):
    \begin{align}
        \dfunc^{\lemmaNodeSymb}(\dv) &= -\conj{\lossfunc}(-\dv) - \conj{(\regfunc^{\lemmaNodeSymb})}(\dv) 
        \nonumber \\
        &= -\conj{\lossfunc}(-\dv) - \sum_{j \in \lemmaSetone} (\conj{\pertfunc}(\transp{\atom{j}}\dv) - \reg) - \sum_{j \in \lemmaSetnone} \pospart{\conj{\pertfunc}(\transp{\atom{j}}\dv) - \reg}
        \label{eq:proof:expansion dual function}
    \end{align}
    where the first equality holds by definition and the second follows from \Cref{prop:dual-relaxregfunc}.

    Let \(\nodeSymb = (\setzero,\setone,\setnone)\) be such that \(\setnone\neq\emptyset\) and let \(\nodeSymb' = (\setzero',\setone',\setnone')\) be a direct successor of \(\nodeSymb\). Then, there exists \(\idxentry\in\setnone\) such that one of the following situation holds:
    \begin{enumerate}[label={\itshape S-\arabic*)}]
        \item \((\setzero', \setone', \setnone') = (\setzero \cup\{\idxentry\}, \setone, \setnone\setminus\{\idxentry\})\)
        \label{item:proof:dual function decomposition: case 1}
        \item \((\setzero', \setone', \setnone') = (\setzero, \setone\cup\{\idxentry\}, \setnone\setminus\{\idxentry\})\).
        \label{item:proof:dual function decomposition: case 2}
    \end{enumerate}
    Applying~\eqref{eq:proof:expansion dual function} to $\nodeSymb'$, we first have for any \(\dv \in \kR^{\ddim}\): 
    {
        \newlength\sumd
        \settowidth{\sumd}{\(\scriptstyle j \in \setnone\setminus\{\idxentry\}\)}
        \begin{align} 
            \dfunc^{\nodeSymb'}(\dv) 
            \,=\,& -\conj{\lossfunc}(-\dv) - \sum_{j \in \setone'} (\conj{\pertfunc}(\transp{\atom{j}}\dv) - \reg) - \sum\limits_{\makebox[\sumd]{\(\scriptstyle j \in \setnone'\)}} \pospart{\conj{\pertfunc}(\transp{\atom{j}}\dv) - \reg} 
            \nonumber \\
            \,=\,& -\conj{\lossfunc}(-\dv) - \sum_{j \in \setone'} (\conj{\pertfunc}(\transp{\atom{j}}\dv) - \reg) - {\sum_{j \in \setnone\setminus\{\idxentry\}}} \pospart{\conj{\pertfunc}(\transp{\atom{j}}\dv) - \reg} 
            \nonumber \\
            \,=\,& -\conj{\lossfunc}(-\dv) - \sum_{j \in \setone'} (\conj{\pertfunc}(\transp{\atom{j}}\dv) - \reg) - \sum\limits_{\makebox[\sumd]{\(\scriptstyle j \in \setnone\)}} \pospart{\conj{\pertfunc}(\transp{\atom{j}}\dv) - \reg} 
            + \pospart{\conj{\pertfunc}(\transp{\atom{\idxentry}}\dv) - \reg}
            \label{eq:proof:dual-link:1:1}
        \end{align}
        where we have used the fact that \(\setnone'=\setnone\setminus\{\idxentry\}\) in both cases~\ref{item:proof:dual function decomposition: case 1} and~\ref{item:proof:dual function decomposition: case 2} to obtain the second equality.
        We now address the two cases separately.
    }

    In case~\ref{item:proof:dual function decomposition: case 1}, we have \(\setone'=\setone\) and~\eqref{eq:proof:dual-link:1:1} becomes
    \begin{align}
        \dfunc^{\nodeSymb'}(\dv) 
        &\,=\, -\conj{\lossfunc}(-\dv) - \sum_{j \in \setone} (\conj{\pertfunc}(\transp{\atom{j}}\dv) - \reg) - \sum\limits_{\makebox[\sumd]{\(\scriptstyle j \in \setnone\)}} \pospart{\conj{\pertfunc}(\transp{\atom{j}}\dv) - \reg} 
        + \pospart{\conj{\pertfunc}(\transp{\atom{\idxentry}}\dv) - \reg}
        \nonumber \\
        &\,=\, \dfunc^{\nodeSymb}(\dv)  + \pospart{\conj{\pertfunc}(\transp{\atom{\idxentry}}\dv) - \reg}
        \label{eq:proof:dual-link:1:3}
    \end{align}
    where we have applied~\eqref{eq:proof:expansion dual function} to obtain the second equality.

    In case~\ref{item:proof:dual function decomposition: case 2}, we have \(\setone'=\setone\cup\{\idxentry\}\) and~\eqref{eq:proof:dual-link:1:1} becomes
    \begin{align}
        \dfunc^{\nodeSymb'}(\dv) 
        &\,=\, -\conj{\lossfunc}(-\dv) - \sum_{j \in \setone\cup\{\idxentry\}} (\conj{\pertfunc}(\transp{\atom{j}}\dv) - \reg) - \sum\limits_{\makebox[\sumd]{\(\scriptstyle j \in \setnone\)}} \pospart{\conj{\pertfunc}(\transp{\atom{j}}\dv) - \reg} 
        + \pospart{\conj{\pertfunc}(\transp{\atom{\idxentry}}\dv) - \reg}
        \nonumber \\
        &\,=\, -\conj{\lossfunc}(-\dv) - \sum_{j \in \setone} (\conj{\pertfunc}(\transp{\atom{j}}\dv) - \reg) - \sum\limits_{\makebox[\sumd]{\(\scriptstyle j \in \setnone\)}} \pospart{\conj{\pertfunc}(\transp{\atom{j}}\dv) - \reg} 
        + \pospart{\conj{\pertfunc}(\transp{\atom{\idxentry}}\dv) - \reg} - (\conj{\pertfunc}(\transp{\atom{\idxentry}}\dv) - \reg)
        \nonumber \\
        &\,=\, \dfunc^{\nodeSymb}(\dv) + \pospart{\conj{\pertfunc}(\transp{\atom{\idxentry}}\dv) - \reg} - (\conj{\pertfunc}(\transp{\atom{\idxentry}}\dv) - \reg)
        \nonumber \\
        &\,=\, \dfunc^{\nodeSymb}(\dv) + \pospart{\reg - \conj{\pertfunc}(\transp{\atom{\idxentry}}\dv)}
        \label{eq:proof:dual-link:2:3}
    \end{align}
    where the last two equalities follow respectively from~\eqref{eq:proof:expansion dual function} and the property $\pvi{} = \pospart{\pvi{}} - \pospart{-\pvi{}}$ for all $\pvi{} \in \kR$.
    
    Gathering the results given in~\eqref{eq:proof:dual-link:1:3}-\eqref{eq:proof:dual-link:2:3} and using the definition of \(\pivot{0}\) and \(\pivot{1}\) in~\eqref{eq:pivot-zero}-\eqref{eq:pivot-one}, one obtains that \(\dfunc^{\nodeSymb'}(\dv)\) satisfies~\eqref{eq:lemma:dual-link direct successor}. 
\end{proof}
}

\subsection{Proof of \texorpdfstring{\Cref{prop:propagation-tests}}{Proposition~\ref{prop:propagation-tests}}}
\label{sec:proof:propagation-tests}

{
\newcommand{\binarynumber}{b}

Let \(\nodeSymb'=(\setzero',\setone',\setnone')\) be a successor of \(\nodeSymb=(\setzero,\setone,\setnone)\) and \(\dv\in\kR^\ddim\).
\Cref{prop:propagation-tests} is a direct consequence of the following two inequalities
\begin{subequations}
    \begin{align}
        \dfunc^{\nodeSymb'_{0,\idxentry}}(\dv) \,\geq\,& \dfunc^{\succnode{0}{\idxentry}}(\dv) \label{eq:proof propagation test:corner inequality a}\\
        \dfunc^{\nodeSymb'_{1,\idxentry}}(\dv) \,\geq\,& \dfunc^{\succnode{1}{\idxentry}}(\dv) \label{eq:proof propagation test:corner inequality b}.
    \end{align}
\end{subequations}
We thus establish~\eqref{eq:proof propagation test:corner inequality a} and~\eqref{eq:proof propagation test:corner inequality b} in the remaining of the section.

Note first that a direct consequence of \Cref{prop:dual-link} is 
\begin{equation}\label{eq:increase-dual-fun}
    \dfunc^{\nodeSymb'}(\dv)\geq \dfunc^{\nodeSymb}(\dv)
\end{equation}
since all terms \(\{\pivot{0}(\transp{\atom{\idxentry}}\dv)\}_{\idxentry\in\setzero'\setminus \setzero}\) and \(\{\pivot{1}(\transp{\atom{\idxentry}}\dv)\}_{\idxentry\in\setone'\setminus \setone}\) are nonnegative. 

Let \(\binarynumber\in\{0, 1\}\).
Particularizing \Cref{lemma:dual-link direct successor} to \(\nodeSymb=\nodeSymb'\) and $\nodeSymb'=\nodeSymb'_{\binarynumber,\idxentry}$ --the direct successor of $\nodeSymb'$ defined in \eqref{eq:direct-zero}-- one obtains:  
\begin{subequations}
    \begin{alignat}{4}
        \label{eq:proof:propagation-tests:1:1}
        \dfunc^{\nodeSymb'_{\binarynumber,\idxentry}}(\dv) 
        &= \dfunc^{\nodeSymb'}(\dv) &+ \ 
        {\pivot{\binarynumber}(\transp{\atom{\idxentry}}\dv)}
        \\
        &\geq \dfunc^{\nodeSymb}(\dv) &+ \ 
        {\pivot{\binarynumber}(\transp{\atom{\idxentry}}\dv)}
        \\
        \label{eq:proof:propagation-tests:1:2}
        &=\dfunc^{\succnode{\binarynumber}{\idxentry}}(\dv)
    \end{alignat}
\end{subequations}
where the inequality follow from~\eqref{eq:increase-dual-fun} (since $\nodeSymb'$ is a successor of $\nodeSymb$) and the last equality from \Cref{lemma:dual-link direct successor}. This establishes the result.  
}

\section{Supplementary materials related to \texorpdfstring{\Cref{sec:numerics}}{Section~\ref{sec:numerics}}}
\label{sec:supp_numerics}

This section gives supplementary materials to our numerical experiments.

\subsection{Mixed-integer Programming Formulations}
\label{sec:supp_numerics:mip}

In our experiments, problem \eqref{prob:prob} is formulated as a \gls{mip} so that it can be handled by commercial solvers like \texttt{Cplex}, \texttt{Gurobi} and \texttt{Mosek}.
For the problem considered in \Cref{sec:numerics:synthetic} where $\lossfunc(\cdot)$ and $\pertfunc(\cdot)$ are given by \eqref{eq:leastsquares}-\eqref{eq:bigm}, we use the following formulation
\begin{equation}
    \left\{
        \begin{array}{rl}
            \min & \tfrac{1}{2}\norm{\obs-\dic\pv}{2}^2 + \reg\transp{\1}\bv \\
            \text{s.t.} & -\bigM \bv \leq \pv \leq \bigM \bv \\
            & \pv \in \kR^{\pdim}, \ \bv \in \{0,1\}^{\pdim}
        \end{array}
    \right.
\end{equation}
where an additional binary variable $\bv \in \{0,1\}^{\pdim}$ is used to encode the nullity of the entries of the continuous variable $\pv \in \kR^{\pdim}$.
A similar approach can be used to reformulate the problems treated in \Cref{sec:numerics:realworld}.
We refer the reader to \Cref{table:mip} for the \gls{mip} formulation of each of the problems considered in our numerical simulations.

\begin{table}[!t]
    \centering
    \begin{tabular*}{0.8\linewidth}{c|l|l}
        \toprule
        & \multicolumn{1}{c|}{Penalty \eqref{eq:bigml1}} & \multicolumn{1}{c}{Penalty \eqref{eq:bigml2}} \\
        \midrule
        \rotatebox[origin=c]{90}{Loss \eqref{eq:leastsquares-loss}}
        & 
        $
        \left\{
            \begin{array}{rl}
                \min & \tfrac{1}{2}\norm{\obs-\dic\pv}{2}^2 + \reg\transp{\1}\bv + \alpha\transp{\1}\mathbf{s} \\
                \text{s.t.} & \pv \geq -\mathbf{s} \\
                & \pv \leq \mathbf{s} \\
                & \pv \geq -\bigM \bv \\
                & \pv \leq \bigM \bv \\
                & \pv \in \kR^{\pdim}, \ \bv \in \{0,1\}^{\pdim}, \ \mathbf{s} \in \kR^{\pdim}
            \end{array}
        \right.
        $
        & 
        $
        \left\{
            \begin{array}{rl}
                \min & \tfrac{1}{2}\norm{\obs-\dic\pv}{2}^2 + \reg\transp{\1}\bv + \alpha\transp{\1}\mathbf{s} \\
                \text{s.t.} & \pv\odot\pv \leq \mathbf{s}\odot\bv \\
                & \pv \geq -\bigM \bv \\ 
                & \pv \leq \bigM \bv \\
                & \pv \in \kR^{\pdim}, \ \bv \in \{0,1\}^{\pdim}, \ \mathbf{s} \in \kR^{\pdim}
            \end{array}
        \right.
        $
        \\
        \midrule
        \rotatebox[origin=c]{90}{Loss \eqref{eq:logistic-loss}} &
        $
        \left\{
            \begin{array}{rll}
                \min & \transp{\1}\mathbf{u} + \reg\transp{\1}\bv + \alpha\transp{\1}\mathbf{s} \\
                \text{s.t.} & \1 \geq \mathbf{v} + \mathbf{w} \\
                & \mathbf{u} \geq -\log(\mathbf{v}) + \obs \odot \dic\pv \\
                & \mathbf{u} \geq -\log(\mathbf{w}) \\
                & \pv \geq -\mathbf{s} \\
                & \pv \leq \mathbf{s} \\
                & \pv \geq -\bigM \bv \\
                & \pv \leq \bigM \bv \\
                & \pv \in \kR^{\pdim}, \ \bv \in \{0,1\}^{\pdim}, \ \mathbf{s} \in \kR^{\pdim} \\
                & \mathbf{u} \in \kR^{\ddim}, \ \mathbf{v} \in \kR^{\ddim}, \ \mathbf{w} \in \kR^{\ddim}
            \end{array}
        \right.
        $
        & 
        $
        \left\{
            \begin{array}{rll}
                \min & \transp{\1}\mathbf{u} + \reg\transp{\1}\bv + \alpha\transp{\1}\mathbf{s} \\
                \text{s.t.} & \mathbf{v} + \mathbf{w} \leq \1 \\
                & \mathbf{u} \geq -\log(\mathbf{v}) + \obs \odot \dic\pv \\
                & \mathbf{u} \geq -\log(\mathbf{w}) \\
                & \pv \odot \pv \leq \mathbf{s} \odot \bv \\
                & -\bigM \bv \leq \pv \leq \bigM \bv \\
                & \pv \in \kR^{\pdim}, \ \bv \in \{0,1\}^{\pdim}, \ \mathbf{s} \in \kR^{\pdim} \\
                & \mathbf{u} \in \kR^{\ddim}, \ \mathbf{v} \in \kR^{\ddim}, \ \mathbf{w} \in \kR^{\ddim}
            \end{array}
        \right.
        $
        \\
        \midrule
        \rotatebox[origin=c]{90}{Loss \eqref{eq:squaredhinge-loss}} & 
        $
        \left\{
            \begin{array}{rll}
                \min & \norm{\mathbf{w}}{2}^2 + \reg\transp{\1}\bv + \alpha\transp{\1}\mathbf{s} \\
                \text{s.t.} & \mathbf{w} \geq \1 - \obs \odot \dic\pv \\
                & \mathbf{w} \geq \0 \\
                & \pv \geq -\mathbf{s} \\
                & \pv \leq \mathbf{s} \\
                & \pv \geq -\bigM \bv \\
                & \pv \leq \bigM \bv \\
                & \pv \in \kR^{\pdim}, \ \bv \in \{0,1\}^{\pdim} \\
                & \mathbf{s} \in \kR^{\pdim}, \mathbf{w} \in \kR^{\pdim}
            \end{array}
        \right.
        $
        & 
        $
        \left\{
            \begin{array}{rll}
                \min & \norm{\mathbf{w}}{2}^2 + \reg\transp{\1}\bv + \alpha\transp{\1}\mathbf{s} \\
                \text{s.t.} & \mathbf{w} \geq \1 - \obs \odot \dic\pv \\
                & \mathbf{w} \geq \0 \\
                & \pv \odot \pv \leq \mathbf{s} \odot \bv  \\
                & \pv \geq -\bigM \bv \\
                & \pv \leq \bigM \bv \\
                & \pv \in \kR^{\pdim}, \ \bv \in \{0,1\}^{\pdim} \\
                & \mathbf{s} \in \kR^{\pdim}, \mathbf{w} \in \kR^{\pdim}
            \end{array}
        \right.
        $
        \\
        \bottomrule
    \end{tabular*}
    \caption{MIP formulations used Section 4. The vectorial inequalities as well as the function $\log(\cdot)$ are taken component-wise and $\odot$ denotes the Hadamard product.}
    \label{table:mip}
\end{table}

\subsection{Implementation Choices}
\label{sec:supp_numerics:implementation_choices}

Our \gls{bnb} solver follows the standard implementation specified in \Cref{sec:bnb} and explores the tree in a ``depth-first'' fashion, as presented in Sec.~5.2.2 by \cite{locatelli2013global}. 
Given some node $\nodeSymb=(\setzero,\setone,\setnone)$ where the decision tree must be expanded, we select an index $\idxentry \in \setnone$ to create new nodes \eqref{eq:direct-zero}-\eqref{eq:direct-one} as
\begin{equation}
    \label{eq:node-selection}
    \idxentry \in \argmax_{\idxentry' \in \setnone} \abs{\node{\pvi{\idxentry'}}}
    .
\end{equation}
where $\node{\pv}$ is the final iterate of the numerical procedure addressing \eqref{prob:relax-node}. 
We use an approach similar to that considered in \cite{hazimeh2022sparse} to solve numerically problem \eqref{prob:relax-node}. 
More precisely, we solve a sequence of sub-problems defined as
\begin{equation}
    \stepcounter{equation}
    \tag{$\node{\rpb}_{\mathcal{W}}$}
    \label{prob:relax-node-subproblem}
    \left\{
    \begin{array}{rl}
        \inf &\lossfunc(\dic\pv) + 
        \biconj{(\node{\regfunc})}(\pv) 
        \\
        \text{s.t.} & \pvi{\idxentry} = 0 \quad \forall \idxentry \notin \mathcal{W}
    \end{array}
    \right.
\end{equation}
where $\mathcal{W} \subseteq \intervint{1}{\pdim}$ is some  working set.
Each sub-problem \eqref{prob:relax-node-subproblem} is solved by using a classical coordinate descent procedure. 
The size of the working set is $\mathcal{W}$ increased based on Fermat's optimality condition violation until no more violations occur.
More specifically, letting $\pv^{\mathcal{W}} \in \kR^{\pdim}$ be the output of the numerical procedure addressing \eqref{prob:relax-node-subproblem}, we let $\mathcal{W}_{\mathrm{new}} = \mathcal{W} \cup \kset{\idxentry}{0 \notin \transp{\atom{\idxentry}}\subdiff\lossfunc(\dic\pv^{\mathcal{W}}) + \subdiff{\biconj{(\node{\separable{\regfunc}{\idxentry}})}}(\pvi{\idxentry}^{\mathcal{W}})}$ and stop the procedure when $\mathcal{W}_{\mathrm{new}} = \mathcal{W}$.

The simultaneous pruning procedure proposed in \Cref{sec:screening} is performed during the resolution of problems \eqref{prob:relax-node}.  
In view of our discussion in \Cref{sec:screening:implementation}, we consider that one iteration of the solving process of \eqref{prob:relax-node} corresponds to the resolution of one sub-problem \eqref{prob:relax-node-subproblem}. 
Stated otherwise, the iterates $\hat{\pv} \in \kR^{\pdim}$ used to implement our pruning methodology in \eqref{eq:dual-point} are the output of the numerical procedure addressing each sub-problem \eqref{prob:relax-node-subproblem}.

\subsection{Technical Implementation Details}
\label{sec:supp_numerics:implementation}

Given some node $\nodeSymb = (\setzero, \setone, \setnone)$, the \gls{bnb} algorithm requires characterizing the convex biconjugate $\biconj{(\node{\regfunc})}(\cdot)$ associated with $\node{\regfunc}(\cdot)$ to construct relaxation \eqref{prob:relax-node}.
We derive its expression from the parametrization of the convex conjugate $\conj{(\node{\regfunc})}(\cdot)$ given in \Cref{prop:dual-relaxregfunc}.
With equation \eqref{eq:dual-relaxregfunc}, we first observe that $\biconj{(\node{\regfunc})}(\pv)=\sum_{\idxentry=1}^{\pdim}\biconj{(\node{\separable{\regfunc}{\idxentry}})}(\pvi{\idxentry})$ where 
\begin{subequations}
    \begin{align}
        \label{eq:supp_numerics:implementation:1:1}
        \biconj{(\separable{\node{\regfunc}}{\idxentry})}(\pvi{}) &= 
        \begin{cases}
            \textstyle\sup_{\vvi{} \in \kR} \pvi{}\vvi{} - 0 &\text{if} \ \idxentry \in \setzero \\
            \textstyle\sup_{\vvi{} \in \kR} \pvi{}\vvi{} - (\conj{\pertfunc}(\vvi{}) - \reg) &\text{if} \ \idxentry \in \setone \\
            \textstyle\sup_{\vvi{} \in \kR} \pvi{}\vvi{} - \pospart{\conj{\pertfunc}(\vvi{}) - \reg} &\text{if} \ \idxentry \in \setnone
        \end{cases} \\
        \label{eq:supp_numerics:implementation:1:2}
        &= 
        \begin{cases}
            \icvx(\pvi{} = 0) &\text{if} \ \idxentry \in \setzero \\
            \biconj{\pertfunc}(\pvi{}) + \reg &\text{if} \ \idxentry \in \setone \\
            \textstyle\sup_{\vvi{} \in \kR} \pvi{}\vvi{} - \pospart{\conj{\pertfunc}(\vvi{}) - \reg} &\text{if} \ \idxentry \in \setnone
        \end{cases} \\
        \label{eq:supp_numerics:implementation:1:3}
        &= 
        \begin{cases}
            \icvx(\pvi{} = 0) &\text{if} \ \idxentry \in \setzero \\
            \pertfunc(\pvi{}) + \reg &\text{if} \ \idxentry \in \setone \\
            \textstyle\sup_{\vvi{} \in \kR} \pvi{}\vvi{} - \pospart{\conj{\pertfunc}(\vvi{}) - \reg} &\text{if} \ \idxentry \in \setnone
            .
        \end{cases}
    \end{align}
\end{subequations}
Equalities \eqref{eq:supp_numerics:implementation:1:1}-\eqref{eq:supp_numerics:implementation:1:2} follow from Definition~4.1 in \cite{beck2017first}.
Equality \eqref{eq:supp_numerics:implementation:1:3} follows from Theorem~4.8 in \cite{beck2017first} since $\pertfunc(\cdot)$ is proper, closed and convex under hypothesis \ref{hyp:pertfunc}.
We next provide a closed-form expression of the case $\idxentry \in \setnone$ in \eqref{eq:supp_numerics:implementation:1:3} for the three expressions of function $\pertfunc(\cdot)$ considered in our numerical experiments.

\paragraph{Function $\pertfunc(\cdot)$ given by \eqref{eq:bigm}}
With this parametrization, we have
\begin{subequations}
    \begin{align}
        \conj{\pertfunc}(\vvi{}) &= \textstyle\sup_{\pvi{} \in \kR} \vvi{}\pvi{} - \icvx(\abs{\pvi{}} \leq \bigM) \\
        &= \bigM\abs{\vvi{}}
        .
    \end{align}
\end{subequations}
Hence, we deduce that $\conj{\pertfunc}(\vvi{}) \leq \reg \iff \abs{\vvi{}} \leq \reg/\bigM$, which gives
\begin{subequations}
    \begin{alignat}{4}
        \label{eq:proof:relaxregfunc:5:1}
        \textstyle\sup_{\vvi{} \in \kR} \pvi{}\vvi{} - \pospart{\conj{\pertfunc}(\vvi{}) - \reg} &= \max\big\{
            \textstyle\sup_{\abs{\vvi{}} \leq \reg/\bigM} \pvi{}\vvi{} \ &&, \ 
            \textstyle\sup_{\abs{\vvi{}} \geq \reg/\bigM} \pvi{}\vvi{} - (\bigM\abs{\vvi{}} - \reg)
        &&\big\} \\
        \label{eq:proof:relaxregfunc:5:2}
        &= \max\big\{
            (\reg/\bigM)\abs{\pvi{}} 
            \ &&, \ 
            \textstyle\sup_{\abs{\vvi{}} \geq \reg/\bigM} \pvi{}\vvi{} - \bigM\abs{\vvi{}} + \reg
        &&\big\}
        .
    \end{alignat}
\end{subequations}
Further, we remark that the supremum in the right member of \eqref{eq:proof:relaxregfunc:5:2} lies in $\kset{\vvi{} \in \kR \cup \{+\infty\}}{\vvi{} \geq 0}$ when $\pvi{} \geq 0$ and in $\kset{\vvi{} \in \kR \cup \{-\infty\}}{\vvi{} \leq 0}$ when $\pvi{} \leq 0$.
Therefore, we obtain
\begin{subequations}
    \begin{alignat}{4}
        \label{eq:proof:relaxregfunc:5:3}
        \textstyle\sup_{\abs{\vvi{}} \geq \reg/\bigM} \pvi{}\vvi{} - \bigM\abs{\vvi{}} + \reg  &= \textstyle\sup_{\abs{\vvi{}} \geq \reg/\bigM} \abs{\vvi{}}(\abs{\pvi{}} - \bigM) + \reg \\
        &= \begin{cases}
            +\infty &\text{if} \ \abs{\pvi{}} > \bigM \\
            (\reg/\bigM)\abs{\pvi{}} &\text{if} \ \abs{\pvi{}} \leq \bigM
        \end{cases} \\
        \label{eq:proof:relaxregfunc:5:4}
        &= (\reg/\bigM)\abs{\pvi{}} + \icvx(\abs{\pvi{}} \leq \bigM)
        .
    \end{alignat}
\end{subequations}
Overall, combining \eqref{eq:proof:relaxregfunc:5:1}-\eqref{eq:proof:relaxregfunc:5:2} with \eqref{eq:proof:relaxregfunc:5:3}-\eqref{eq:proof:relaxregfunc:5:4} gives
\begin{subequations}
    \begin{align}
        \biconj{(\node{\separable{\regfunc}{\idxentry}})}(\pvi{}) &= \textstyle\sup_{\vvi{} \in \kR} \pvi{}\vvi{} - \pospart{\conj{\pertfunc}(\vvi{}) - \reg} \\
        &= \max\{(\reg/\bigM)\abs{\pvi{}}, \ (\reg/\bigM)\abs{\pvi{}} + \icvx(\abs{\pvi{}} \leq \bigM)\} \\
        &= (\reg/\bigM)\abs{\pvi{}} + \icvx(\abs{\pvi{}} \leq \bigM)
    \end{align}
\end{subequations}
when $\idxentry \in \setnone$ in \eqref{eq:supp_numerics:implementation:1:3}.

\paragraph{Function $\pertfunc(\cdot)$ given by \eqref{eq:bigml1}}
 
With this parametrization, we have
\begin{subequations}
    \begin{align}
        \conj{\pertfunc}(\vvi{}) &= \textstyle\sup_{\abs{\pvi{}} \leq \bigM} \vvi{}\pvi{} - \alpha\abs{\pvi{}} \\
        &= \textstyle\sup_{\abs{\pvi{}} \leq \bigM} \abs{\pvi{}}(\abs{\vvi{}} - \alpha) \\
        &= \begin{cases}
            0 &\text{if} \ \abs{\vvi{}} \leq \alpha \\
            \bigM(\abs{\vvi{}} - \alpha) &\text{if} \ \abs{\vvi{}} > \alpha
        \end{cases} \\
        \label{eq:proof:relaxregfunc:6}
        &= \bigM\pospart{\abs{\vvi{}} - \alpha}
        .
    \end{align}
\end{subequations}
Hence, we deduce that $\conj{\pertfunc}(\vvi{}) \leq \reg \iff \abs{\vvi{}} \leq \alpha+\reg/\bigM$, which gives
\begin{subequations}
    \begin{alignat}{4}
        \label{eq:proof:relaxregfunc:7:1a}
        \textstyle\sup_{\vvi{} \in \kR} \pvi{}\vvi{} - \pospart{\conj{\pertfunc}(\vvi{}) - \reg} 
        &= \max\big\{
            \textstyle\sup_{\abs{\vvi{}} \leq \alpha + \reg/\bigM} \pvi{}\vvi{} \ &&, \ 
            \textstyle\sup_{\abs{\vvi{}} \geq \alpha + \reg/\bigM} \pvi{}\vvi{} - \bigM\abs{\vvi{}} + \bigM\alpha + \reg
        &&\big\} \\
        \label{eq:proof:relaxregfunc:7:2}
        &= \max\big\{
            (\alpha + \reg/\bigM)\abs{\pvi{}} \ &&, \ 
            \textstyle\sup_{\abs{\vvi{}} \geq \alpha + \reg/\bigM} \pvi{}\vvi{} - \bigM\abs{\vvi{}} + \bigM\alpha + \reg
        &&\big\} 
    \end{alignat}
\end{subequations}
Further, we remark that the supremum of the right member in \eqref{eq:proof:relaxregfunc:7:2} lies in $\kset{\vvi{} \in \kR \cup \{+\infty\}}{\vvi{} \geq 0}$ when $\pvi{} \geq 0$ and in $\kset{\vvi{} \in \kR \cup \{-\infty\}}{\vvi{} \leq 0}$ when $\pvi{} \leq 0$.
Therefore, we obtain
\begin{subequations}
    \begin{align}
        \label{eq:proof:relaxregfunc:7:4}
        \textstyle\sup_{\abs{\vvi{}} \geq \alpha + \reg/\bigM} \pvi{}\vvi{} - \bigM\abs{\vvi{}} + \bigM\alpha + \reg &= \textstyle\sup_{\abs{\vvi{}} \geq \alpha + \reg/\bigM} \abs{\vvi{}}(\abs{\pvi{}} - \bigM) + \bigM\alpha + \reg \\
        &= 
        \begin{cases}
            +\infty &\text{if} \ \abs{\pvi{}} > \bigM \\
            (\alpha + \reg/\bigM)\abs{\pvi{}} &\text{if} \ \abs{\pvi{}} \leq \bigM
        \end{cases} \\
        \label{eq:proof:relaxregfunc:7:5}
        &= (\alpha + \reg/\bigM)\abs{\pvi{}} + \icvx(\abs{\pvi{}} \leq \bigM)
        .
    \end{align}
\end{subequations}
Overall, combining \eqref{eq:proof:relaxregfunc:7:1a}-\eqref{eq:proof:relaxregfunc:7:2} with \eqref{eq:proof:relaxregfunc:7:4}-\eqref{eq:proof:relaxregfunc:7:5} gives
\begin{subequations}
    \begin{align}
        \biconj{(\node{\separable{\regfunc}{\idxentry}})}(\pvi{}) &= \textstyle\sup_{\vvi{} \in \kR} \pvi{}\vvi{} - \pospart{\conj{\pertfunc}(\vvi{}) - \reg} \\
        &= \max\big\{
            (\alpha + \reg/\bigM)\abs{\pvi{}}, \ 
            (\alpha + \reg/\bigM)\abs{\pvi{}} + \icvx(\abs{\pvi{}} \leq \bigM)
        \big\} \\
        &= (\alpha + \reg/\bigM)\abs{\pvi{}} + \icvx(\abs{\pvi{}} \leq \bigM)
    \end{align}
\end{subequations}
when $\idxentry \in \setnone$ in \eqref{eq:supp_numerics:implementation:1:3}.

\paragraph{Function $\pertfunc(\cdot)$ given by \eqref{eq:bigml2}.}
With this parametrization, we obtain
\begin{subequations}
    \begin{align}
        \conj{\pertfunc}(\vvi{}) 
        &\triangleq \textstyle\sup_{\abs{\pvi{}} \leq \bigM} \vvi{}\pvi{} - \alpha\pvi{}^2 \\
        &= -\textstyle\inf_{\abs{\pvi{}} \leq \bigM} \alpha\pvi{}^2 - \vvi{}\pvi{} 
        \label{eq:proof:relaxregfunc:square abs + interval:buf conjugate} \\
        &= \begin{cases}
            \tfrac{\vvi{}^2}{4\alpha} &\text{if} \ \abs{\vvi{}} \leq 2\alpha\bigM \\
            \bigM\abs{\vvi{}} - \alpha\bigM^2 &\text{if} \ \abs{\vvi{}} > 2\alpha\bigM
        \end{cases}
    \end{align}
\end{subequations}
by noting that the scalar defined as the orthogonal projection   of \(\tfrac{\vvi{}}{2\alpha}\) onto the interval \([-\bigM, \bigM]\) satisfies the necessary and sufficient first order optimality condition (see \textit{e.g.} Corollary 3.68 in~\cite{beck2017first}) associated to the convex minimization problem involved in~\eqref{eq:proof:relaxregfunc:square abs + interval:buf conjugate}.
The function $\conj{\pertfunc}(\cdot)$ is continuous, monotone, minimized at $\vvi{} = 0$ and one has $\conj{\pertfunc}(\vvi{}) = \alpha\bigM^2$ at the threshold $\abs{\vvi{}} = 2\alpha\bigM$.
In this view, we deduce that
\begin{subequations}
    \begin{align}
        \conj{\pertfunc}(\vvi{}) \leq \reg &\iff 
        \begin{cases}
            \bigM\abs{\vvi{}} - \alpha\bigM^2 \leq \reg &\text{if} \ \alpha\bigM^2 \leq \reg  \\
            \tfrac{\vvi{}^2}{4\alpha} \leq \reg &\text{if} \ \alpha\bigM^2 > \reg
        \end{cases}
        \\
        \label{eq:proof:relaxregfunc:8}
        &\iff \begin{cases}
            \abs{\vvi{}} \leq \alpha\bigM + \reg/\bigM & \text{if} \ \bigM \leq \sqrt{\reg/\alpha} \\
            \abs{\vvi{}} \leq 2\sqrt{\reg\alpha} &\text{if} \ \bigM > \sqrt{\reg/\alpha}
            .
        \end{cases}
    \end{align}
\end{subequations}
By treating the two above cases separately, we next show that 
\begin{subequations}
    \begin{align}
        \biconj{(\node{\separable{\regfunc}{\idxentry}})}(\pvi{}) &= \textstyle\sup_{\vvi{} \in \kR} \pvi{}\vvi{} - \pospart{\conj{\pertfunc}(\vvi{}) - \reg} \\
        &=
        \begin{cases}
            \icvx(\abs{\pvi{}} \leq \bigM) + (\tfrac{\reg}{\bigM} + \alpha\bigM)\abs{\pvi{}} &\text{if} \ \bigM \leq \sqrt{\reg/\alpha} \\
            \icvx(\abs{\pvi{}} \leq \bigM) + 2\reg B(\abs{\pvi{}}\sqrt{\alpha/\reg}) &\text{if} \ \bigM > \sqrt{\reg/\alpha}
        \end{cases}
    \end{align}
\end{subequations}
when $\idxentry \in \setnone$ in \eqref{eq:supp_numerics:implementation:1:3}.

\textit{\textul{Case $\bigM \leq \sqrt{\reg/\alpha}$.}} 
We have
\begin{subequations}
    \begin{alignat}{4}
        \label{eq:proof:relaxregfunc:9:1}
        \textstyle\sup_{\vvi{} \in \kR} \pvi{}\vvi{} - \pospart{\conj{\pertfunc}(\vvi{}) - \reg} &= \max\big\{
            \textstyle\sup_{\abs{\vvi{}} \leq \alpha\bigM + \reg/\bigM} \pvi{}\vvi{} \ &&, \ 
            \textstyle\sup_{\abs{\vvi{}} \geq \alpha\bigM + \reg/\bigM} \pvi{}\vvi{} - \tfrac{\vvi{}^2}{4\alpha} +  \alpha\pospart{\tfrac{\abs{\vvi{}}}{2\alpha} - \bigM}^2 + \reg
        &&\big\} \\
        \label{eq:proof:relaxregfunc:9:2}
        &= \max\big\{
            \textstyle\sup_{\abs{\vvi{}} \leq \alpha\bigM + \reg/\bigM} \pvi{}\vvi{} \ &&, \ 
            \textstyle\sup_{\abs{\vvi{}} \geq \alpha\bigM + \reg/\bigM} \pvi{}\vvi{} - \tfrac{\vvi{}^2}{4\alpha} +  \alpha(\tfrac{\abs{\vvi{}}}{2\alpha} - \bigM)^2 + \reg
        &&\big\} \\
        \label{eq:proof:relaxregfunc:9:3}
        &= \max\big\{
            (\alpha\bigM + \reg/\bigM)\abs{\pvi{}} \ &&, \ 
            \textstyle\sup_{\abs{\vvi{}} \geq \alpha\bigM + \reg/\bigM} \pvi{}\vvi{} - \bigM\abs{\vvi{}} + \alpha\bigM^2 + \reg
        &&\big\} 
    \end{alignat}
\end{subequations}
where equality \eqref{eq:proof:relaxregfunc:9:2} holds since $\bigM \leq \sqrt{\reg/\alpha} \implies \alpha\bigM + \reg\bigM \geq 2\alpha\bigM$.
Further, we remark that the supremum of the right member in \eqref{eq:proof:relaxregfunc:9:3} lies in $\kset{\vvi{} \in \kR \cup \{+\infty\}}{\vvi{} \geq 0}$ when $\pvi{} \geq 0$ and in $\kset{\vvi{} \in \kR \cup \{-\infty\}}{\vvi{} \leq 0}$ when $\pvi{} \leq 0$.
Therefore, we obtain
\begin{subequations}
    \begin{align}
        \label{eq:proof:relaxregfunc:9:4}
        \textstyle\sup_{\abs{\vvi{}} \geq \alpha\bigM + \reg/\bigM} \pvi{}\vvi{} - \bigM\abs{\vvi{}} + \alpha\bigM^2 + \reg &= \textstyle\sup_{\abs{\vvi{}} \geq \alpha\bigM + \reg/\bigM} \abs{\vvi{}}(\abs{\pvi{}} - \bigM) + \alpha\bigM^2 + \reg \\
        &= 
        \begin{cases}
            +\infty &\text{if} \ \abs{\pvi{}} > \bigM \\
            (\alpha\bigM + \reg/\bigM)\abs{\pvi{}} &\text{if} \ \abs{\pvi{}} \leq \bigM
        \end{cases} \\
        \label{eq:proof:relaxregfunc:9:5}
        &= (\alpha\bigM + \reg/\bigM)\abs{\pvi{}} + \icvx(\abs{\pvi{}} \leq \bigM)
        .
    \end{align}
\end{subequations}
Overall, combining \eqref{eq:proof:relaxregfunc:9:1}-\eqref{eq:proof:relaxregfunc:9:3} with \eqref{eq:proof:relaxregfunc:9:4}-\eqref{eq:proof:relaxregfunc:9:5} gives
\begin{subequations}
    \begin{align}
        \biconj{(\node{\separable{\regfunc}{\idxentry}})}(\pvi{}) &= \max\big\{
            (\alpha\bigM + \reg/\bigM)\abs{\pvi{}}, \ 
            (\alpha\bigM + \reg/\bigM)\abs{\pvi{}} + \icvx(\abs{\pvi{}} \leq \bigM)
        \big\} \\
        \label{eq:proof:relaxregfunc:9:6}
        &= (\alpha\bigM + \reg/\bigM)\abs{\pvi{}} + \icvx(\abs{\pvi{}} \leq \bigM)
    \end{align}
\end{subequations}
for the case $\bigM \leq \sqrt{\reg/\alpha}$.

\textit{\textul{Case $\bigM > \sqrt{\reg/\alpha}.$}} 
We have
\begin{subequations}
    \begin{alignat}{4}
        \label{eq:proof:relaxregfunc:10:1}
        \textstyle\sup_{\vvi{} \in \kR} \pvi{}\vvi{} - \pospart{\conj{\pertfunc}(\vvi{}) - \reg} &= \max\big\{
            \textstyle\sup_{\abs{\vvi{}} \leq 2\sqrt{\reg\alpha}} \pvi{}\vvi{} \ &&, \ 
            \textstyle\sup_{\abs{\vvi{}} \geq 2\sqrt{\reg\alpha}} \pvi{}\vvi{} - \tfrac{\vvi{}^2}{4\alpha} + \alpha\pospart{\tfrac{\abs{\vvi{}}}{2\alpha} - \bigM}^2 + \reg
        &&\big\} \\
        \label{eq:proof:relaxregfunc:10:2}
        &= \max\big\{
            2\sqrt{\reg\alpha}\abs{\pvi{}} \ &&, \ 
            \textstyle\sup_{\abs{\vvi{}} \geq 2\sqrt{\reg\alpha}} \pvi{}\vvi{} - \tfrac{\vvi{}^2}{4\alpha} + \alpha\pospart{\tfrac{\abs{\vvi{}}}{2\alpha} - \bigM}^2 + \reg
        &&\big\}
        .
    \end{alignat}
\end{subequations}
Further, we note that since $\bigM > \sqrt{\reg/\alpha} \implies 2\alpha\bigM > 2\sqrt{\reg\alpha}$, the right member in equation \eqref{eq:proof:relaxregfunc:10:2} splits into
\begin{subequations}
    \begin{alignat}{4}
        \label{eq:proof:relaxregfunc:11:0}
        & \textstyle\sup_{\abs{\vvi{}} \geq 2\sqrt{\reg\alpha}} \pvi{}\vvi{} - \tfrac{\vvi{}^2}{4\alpha} + \alpha\pospart{\tfrac{\abs{\vvi{}}}{2\alpha} - \bigM}^2 + \reg \\
        =& \max\big\{
            \textstyle\sup_{2\sqrt{\reg\alpha} \leq \abs{\vvi{}} \leq 2\alpha\bigM} \pvi{}\vvi{} - \tfrac{\vvi{}^2}{4\alpha} + \alpha\pospart{\tfrac{\abs{\vvi{}}}{2\alpha} - \bigM}^2 + \reg 
            \ &&, \ 
            \textstyle\sup_{\abs{\vvi{}} \geq 2\alpha\bigM} \pvi{}\vvi{} - \tfrac{\vvi{}^2}{4\alpha} + \alpha\pospart{\tfrac{\abs{\vvi{}}}{2\alpha} - \bigM}^2 + \reg
        &&\big\} \\
        \label{eq:proof:relaxregfunc:11:1}
        =& \max\big\{
            \textstyle\sup_{2\sqrt{\reg\alpha} \leq \abs{\vvi{}} \leq 2\alpha\bigM} \pvi{}\vvi{} - \tfrac{\vvi{}^2}{4\alpha} + \reg 
            \ &&, \ 
            \textstyle\sup_{\abs{\vvi{}} \geq 2\alpha\bigM} \pvi{}\vvi{} - \tfrac{\vvi{}^2}{4\alpha} + \alpha(\tfrac{\abs{\vvi{}}}{2\alpha} - \bigM)^2 + \reg
        &&\big\} \\
        \label{eq:proof:relaxregfunc:11:2}
        =& \max\big\{
            \textstyle\sup_{2\sqrt{\reg\alpha} \leq \abs{\vvi{}} \leq 2\alpha\bigM} \pvi{}\vvi{} - \tfrac{\vvi{}^2}{4\alpha} + \reg 
            \ &&, \ 
            \textstyle\sup_{\abs{\vvi{}} \geq 2\alpha\bigM} \pvi{}\vvi{} - \bigM\abs{\vvi{}} + \alpha\bigM^2 + \reg
        &&\big\} 
        .
    \end{alignat}
\end{subequations}
On the one hand, the left member in equation \eqref{eq:proof:relaxregfunc:11:2} can be expressed in closed-form as
\begin{align}
    \textstyle\sup_{2\sqrt{\reg\alpha} \leq \abs{\vvi{}} \leq 2\alpha\bigM} \pvi{}\vvi{} - \tfrac{\vvi{}^2}{4\alpha} + \reg 
    &=
    \reg
    -\textstyle\inf_{2\sqrt{\reg\alpha} \leq \abs{\vvi{}} \leq 2\alpha\bigM} \tfrac{\vvi{}^2}{4\alpha} - \pvi{}\vvi{}
    \label{eq:proof:relaxregfunc:square abs + interval:buf biconjugate}
    \\
    &= \begin{cases}
        2\sqrt{\reg\alpha} \abs{\pvi{}} &\text{if} \ \abs{\pvi{}} \leq \sqrt{\reg/\alpha} \\
        \alpha\pvi{}^2 + \reg &\text{if} \ \abs{\pvi{}} \in \ ]\sqrt{\reg/\alpha}, \bigM] \\
        2\alpha\bigM\abs{\pvi{}} - \alpha\bigM^2 + \reg &\text{if} \ \abs{\pvi{}} > \bigM
    \end{cases}
    \label{eq:proof:relaxregfunc:12:1}
\end{align}
by noting that the scalar defined as the orthogonal projection of \(2\alpha\pvi{}\) onto the interval \([2\sqrt{\reg\alpha},2\alpha\bigM]\) satisfies the necessary and sufficient first order optimality condition (see \textit{e.g.} Corollary 3.68 in~\cite{beck2017first}) associated to the convex minimization problem involved in the right-hand side of~\eqref{eq:proof:relaxregfunc:square abs + interval:buf biconjugate}.
On the other hand, we remark that the supremum of the right member in equation \eqref{eq:proof:relaxregfunc:11:2} lies in $\kset{\vvi{} \in \kR \cup \{+\infty\}}{\vvi{} \geq 0}$ when $\pvi{} \geq 0$ and in $\kset{\vvi{} \in \kR \cup \{-\infty\}}{\vvi{} \leq 0}$ when $\pvi{} \leq 0$, which yields
\begin{subequations}
    \begin{align}
        \label{eq:proof:relaxregfunc:12:2}
        \textstyle\sup_{\abs{\vvi{}} \geq 2\alpha\bigM} \pvi{}\vvi{} - \bigM\abs{\vvi{}} + \alpha\bigM^2 + \reg &= \textstyle\sup_{\abs{\vvi{}} \geq 2\alpha\bigM} \abs{\vvi{}}(\abs{\pvi{}} - \bigM) + \alpha\bigM^2 + \reg \\
        &= 
        \begin{cases}
            +\infty &\text{if} \ \abs{\pvi{}} > \bigM \\
            2\alpha\bigM\abs{\pvi{}} - \alpha\bigM^2 + \reg &\text{if} \ \abs{\pvi{}} \leq \bigM
        \end{cases} \\
        \label{eq:proof:relaxregfunc:12:3}
        &= 2\alpha\bigM\abs{\pvi{}} - \alpha\bigM^2 + \reg + \icvx(\abs{\pvi{}} \leq \bigM)
        .
    \end{align}
\end{subequations}
By plugging \eqref{eq:proof:relaxregfunc:12:1} and \eqref{eq:proof:relaxregfunc:12:2}-\eqref{eq:proof:relaxregfunc:12:3} into each member of \eqref{eq:proof:relaxregfunc:11:0}-\eqref{eq:proof:relaxregfunc:11:2}, we deduce that
\begin{subequations}
    \begin{alignat}{4}
        \label{eq:proof:relaxregfunc:13:1}
        & \textstyle\sup_{\abs{\vvi{}} \geq 2\sqrt{\reg\alpha}} \pvi{}\vvi{} - \tfrac{\vvi{}^2}{4\alpha} + \alpha\pospart{\tfrac{\abs{\vvi{}}}{2\alpha} - \bigM}^2 + \reg \\
        \label{eq:proof:relaxregfunc:13:2}
        =& 
        \begin{cases}
            \max\big\{
                2\sqrt{\reg\alpha}\abs{\pvi{}}
                \qquad\qquad\quad \, , \
                2\alpha\bigM\abs{\pvi{}} - \alpha\bigM^2 + \reg
            \big\} &\text{if} \ \abs{\pvi{}} \leq \sqrt{\reg/\alpha} \\
            \max\big\{
                \alpha\pvi{}^2 + \reg
                \qquad\qquad\quad \ \, , \
                2\alpha\bigM\abs{\pvi{}} - \alpha\bigM^2 + \reg
            \big\} &\text{if} \ \abs{\pvi{}} \in \ ]\sqrt{\reg/\alpha},\bigM] \\
            \max\big\{
                2\alpha\bigM\abs{\pvi{}} - \alpha\bigM^2 + \reg 
                \ , \
                +\infty
            \qquad\qquad\qquad \ \ \ \, \big\} &\text{if} \ \abs{\pvi{}} > \bigM \\
        \end{cases} \\
        \label{eq:proof:relaxregfunc:13:3}
        =& 
        \begin{cases}
            2\sqrt{\reg\alpha}\abs{\pvi{}} &\text{if} \ \abs{\pvi{}} \leq \sqrt{\reg/\alpha} \\
            \alpha\pvi{}^2 + \reg &\text{if} \ \abs{\pvi{}} \in \ ]\sqrt{\reg/\alpha},\bigM] \\
            +\infty &\text{if} \ \abs{\pvi{}} > \bigM \\
        \end{cases}
    \end{alignat}
\end{subequations}
where equality \eqref{eq:proof:relaxregfunc:13:2} holds since $2\sqrt{\reg\alpha} \leq 2\alpha\bigM$ and $-\alpha\bigM^2 + \reg \leq 0$, reminding that we consider the case $\bigM > \sqrt{\reg/\alpha}$.
Overall, combining \eqref{eq:proof:relaxregfunc:10:1}-\eqref{eq:proof:relaxregfunc:10:2} with \eqref{eq:proof:relaxregfunc:13:1}-\eqref{eq:proof:relaxregfunc:13:3} gives
\begin{subequations}
    \begin{align}
        \biconj{(\node{\separable{\regfunc}{\idxentry}})}(\pvi{}) &=
        \begin{cases}
            2\sqrt{\reg\alpha}\abs{\pvi{}} &\text{if} \ \abs{\pvi{}} \leq \sqrt{\reg/\alpha} \\
            \alpha\pvi{}^2 + \reg &\text{if} \ \abs{\pvi{}} \in \ ]\sqrt{\reg/\alpha},\bigM] \\
            +\infty &\text{if} \ \abs{\pvi{}} > \bigM \\
        \end{cases} \\
        \label{eq:proof:relaxregfunc:14}
        &= 2\reg B(\abs{\pvi{}}\sqrt{\alpha/\reg}) + \icvx(\abs{\pvi{}} \leq \bigM)
    \end{align}
\end{subequations}
for the case $\bigM > \sqrt{\reg/\alpha}$.

\subsection{Hyperparameters Calibration}
\label{sec:supp_numerics:hyperparameters}

To calibrate the value of $\reg$ in \eqref{prob:prob} and the hyperparameters of the function $\pertfunc(\cdot)$, we use the \texttt{l0learn} package \cite{dedieu2021learning} that can approximately solve some specific instances of the problem.
More specifically, we call the \texttt{cv.fit} procedure to perform a grid search over the values of $\reg$ and the hyperparameters of the function $\pertfunc(\cdot)$.
For each point in the grid, an approximate solution $\hat{\pv}$ to problem \eqref{prob:prob} is constructed and an associated cross-validation score is computed.
We select the value of $\reg$ and the hyperparameters in $\pertfunc(\cdot)$ leading to the best cross-validation score.
For the synthetic instances considered in \Cref{sec:numerics:synthetic}, we only consider the candidates $\hat{\pv}$ in the grid with the best F1-score for the support recovery of the ground truth $\groundtruth$.

\end{document}